\documentclass[a4paper,12pt]{article}

\usepackage[english]{babel}
\usepackage[T1]{fontenc}
\usepackage{amssymb}
\usepackage{amsmath}
\usepackage{amsthm}
\usepackage{mathrsfs}
\usepackage[all]{xy}

\usepackage{ulem}

\usepackage{sectsty}

\sectionfont{\fontsize{14}{15}\selectfont}

\newtheorem{thm}{Theorem}[section]
\newtheorem{lem}[thm]{Lemma}
\newtheorem{cor}[thm]{Corollary}
\newtheorem{prop}[thm]{Proposition}
\newtheorem{question}{Question}
\newtheorem{deff}[thm]{Definition}
\newcommand{\Z}{\mathbb{Z}}

\newcommand{\A}{{\cal A}}

\newcommand{\Oe}{\rm \widehat{OE}}
\def\OE{{\rm OE}}

\newcommand{\F}{{\cal F}}

\newcommand{\cR}{{\cal R}}

\def\aut {{\rm Aut}}
\def\out {{\rm Out}}
\def\inn {{\rm Inn}}
\def\hnn {{\rm HNN}}
\def\Iso {{\rm Iso}}

\def\Z{{\mathbb Z}}
\newtheorem{exa}[thm]{Example}
\newtheorem{rem}[thm]{Remark}

\usepackage{amssymb}
\usepackage{amsmath}
\usepackage{amsthm}
\usepackage{mathrsfs}
\usepackage[all]{xy}
\usepackage{graphicx,color}
\usepackage{authblk}

\title{Profinite genus of HNN-extensions with finite associated subgroups}

\author[$\dagger$]{V.R. de Bessa} \author[$\star$]{A.L.P. Porto} \author[$\ddag$]{P.A. Zalesskii\footnote{The third author was partially supported by CNPq. The second author is grateful for the financial support from FAPDF and the Post-Doctoral internship period at UFMG. }}

\affil[$\dagger$]{Federal University of Viçosa, Rio Paranaíba, Brazil}
\affil[$\star$]{ICT, Federal University of the Jequitinhonha and Mucuri Valleys, Diamantina, Brazil} 
\affil[$\ddag$]{Department of Mathematics, University of Brasilia, Brasilia, Brazil}

\begin{document}

\maketitle

\begin{abstract} 

We study the profinite genus of HNN-extensions whose associated subgroups are finite. We give precise formulas for the number of isomorphism classes of $\hnn(G,H,K,t,f)$ and of its profinite completion and compute the profinite genus of such an HNN-extension $\hnn(G,H,K,t,f)$. We also list various situations when   $\hnn(G,H,K,t,f)$ is determined by its profinite completion. 
\end{abstract}


\section{Introduction} \label{intro}

There has been much recent study of whether residually finite groups, or classes of residually finite groups of combinatorial nature may be distinguished from each other by their sets of finite quotient groups.

In group theory the study in this direction started in 70-th of the last century when Baumslag \cite{BAU74},  Stebe \cite{Ste72} and others found examples of finitely generated residually finite groups having the same set of finite quotients. The general question  addressed in this study can be formulated as follows: 
		
		\begin{question}\label{question}
			To what extent a finitely generated residually finite group $\Gamma$ is  determined by its finite quotients?
		\end{question} 
		
		The study leaded to the notion of genus $\mathfrak{g}(G)$ of a residually finite group $G$, the set of isomorphism classes of finitely generated residually finite groups having the same set of finite quotients as $G$. Equivalently, $\mathfrak{g}(G)$ is the set of isomorphism classes of finitely generated residually finite groups having the  profinite completion isomorphic to the profinite completion $\widehat G$ of $G$. 
		
		This study mostly was concentrated to establish whether the cardinality of  the  genus $\mathfrak{g}(G)$ is finite or  1 (see  \cite{Aka, BCR16, BMRS18, BMRS20, BPZ2, BPZ, BZ, Dav, GPS, GS, GZ, J, J2, J3, M, Sam, Sam2, Wil17} for example) (we use the same term {\it genus} for $g(G)$ from now on). There are only few papers where exact numbers or estimates of the genus appear due to the difficulties of such calculation (see \cite{BPZ2, BPZ, BZ, GZ, Ner19, Ner20, Ner24}).
However, the following principle question of Remeslennikov is still open \cite[Question 5.48]{K}.
		
		\begin{question} (V.N. Remeslennikov)
		Is the genus $g(F)$ of a free group $F$ of finite rank equals 1?
		
		\end{question}

There are two basic constructions in combinatorial group theory: free product with amalgamation and HNN-extension. As many groups of combinatorial and geometric nature split as an amalgamted free product or HNN-extension, it has sense  to study the genus of these constructions.  It is natural to start with splitting over finite groups, and for free products with amalgamation this was done in \cite{BPZ}.  However, since Remeslennikov's question is not answered, it is not clear whether the profinite completion of an one-ended group does not split as an amalgamated free product or an HNN-extension over a finite group. This naturally gives restriction to the family in which one does considerations.

A finitely generated residually finite group $G$ with less than 2 ends, i.e. a finite or one ended group, will be called an $\OE$-group in the paper.   It follows from the famous Stallings theorem that $G$ is an $\OE$-group  if and only if whenever it acts on a tree with finite edge stabilizers
it  has a global fixed point. A finitely generated profinite group will be called $\OE$-group if it has the same property:   whenever it acts   on a  profinite tree  with finite edge stabilizers, then it fixes a vertex.  Note that  finite and one-ended groups acting on a tree have a global fixed point and we need to extend this property to the profinite completion.  A finitely generated residually finite group $G$ will be called $\Oe$-group if  $\widehat G$  is a profinite $\OE$-group, i.e. $\widehat G$  can not act on a profinite tree with finite edge stabilizers without a global fixed point (see Definition \ref{Profinite tree} for the definition of a profinite tree). 

Note that the class $\Oe$-groups is quite large. It contains all finitely generated residually finite not infinite virtually cyclic soluble groups, virtually surface  groups, indecomposable (into free products) 3-manifold groups,  indecomposable RAAGs, as well as all arithmetic groups of rank $>1$ and finitely generated residually finite Fab groups (see Propositions \ref{virtually cyclic} and \ref{examples}).

The subject of the present paper is to complement the work that was done in \cite{BPZ2} and  \cite{BPZ} by studying the profinite genus of HNN-extensions of an $\Oe$-group whose associated subgroups are finite. Note that 
in \cite{BZ} examples of HNN-extensions with finite base group whose profinite completions are isomorphic were given and in this paper we give more examples with not obligatory finite base group.

We first investigate the isomorphism problems for (profinite) abstract HNN-extensions that are of independent interest and in fact calculate the number of isomorpism classes of such abstract HNN-extensions (see Sections \ref{iso abstract HNN} and \ref{profinitehnnextensions}). 

Let $G=\hnn(G_1, H, K, f, t)$ be an abstract HNN-extension with finite associated subgroups $H$ and $K$.  Denote by $S$ the set of all finite subgroups of $G_1$. There is a natural action of $G_1\rtimes \aut(G_1)$ on the set $S$ namely \begin{equation} (g_1,\alpha)\cdot X= g_1\alpha(X)g_1^{-1} \in S, \mbox{ where } (g_1,\alpha) \in G_1\rtimes \aut(G_1), X \in S. \end{equation}

Denote by  $\aut_{G_1}(H)$  the set of automorphisms of $G_1$ that leaves $H$ invariant. 
Then the stabilizer of the conjugacy class of $H$ in $G_1\rtimes \aut(G_1)$ is the subgroup
\begin{equation} \Gamma_H:=G_1 \rtimes \aut_{G_1}(H)\end{equation}  of the holomorph $G_1\rtimes \aut(G_1)$. 
Denote by $\Gamma_{HK}$  the stabilizer of $K\in S$ in $\Gamma_H$. 
 
 Let $\Iso(H, K)$ be the set of all isomorphisms of $H$ in $K.$ Then $\Gamma_{HK}$ acts naturally on the set of isomorphisms $\Iso(H,K)$  whose action is given by $$(g_1,\alpha) \cdot f=\tau_{g_1} \alpha f \alpha^{-1},$$ 
 where $(g_1,\alpha) \in \Gamma_{HK}$,  $f \in \Iso(H, K)$ and $\tau_{g_1}$ it the inner automorphism corresponding to $g_1$.

 We denote by $\widetilde\Gamma_{HK}$ the image of $\Gamma_{HK}$ in the symmetric group $Sym(\Iso(H,K)).$ Depending on the choice  of $H$ and $K$  it might exist $\psi\in \aut(G_1)$  that swaps the conjugacy classes of $H$ and $K$, say $\psi(H)=K$ and $\psi(K)=H^{g_1}$ for some $g_1\in G_1$. In this case the "invertion" of elements of $\Iso(H,K)$ followed by the action by the element $( g', \psi)\in G_1\rtimes \aut(G_1)$ gives an element $\iota\in Sym(\Iso(H,K))$ that normalizes $\widetilde\Gamma_{HK}$ and so $\overline\Gamma_{HK}:=\widetilde\Gamma_{HK} \cdot \langle\iota\rangle$ is a group that does not depend in fact upon a choice of $\psi$ (see Lemma \ref{carmo}). If such $\psi$ does not exist we put $\overline\Gamma_{HK}:=\widetilde\Gamma_{HK}$.

\begin{thm} Let $G_1$ be an $\OE$-group and $$G=\hnn(G_1, H, K, f, t), B=\hnn(G_1, H, K, f_1, t_1)$$ be  abstract HNN-extensions with fixed finite associated  subgroups $H,K$ of $G_1$. Then $G$ and $B$ are isomorphic if and only if $f$ and $f_1$ belong to the same $\overline\Gamma_{HK}$-orbit in $\Iso(H,K)$.    
\end{thm}

\begin{cor}\label{general iso} The number of isomorphism classes of HNN-extensions \break
$\hnn(G_1,H,K,f,t)$  with $G_1,H,K$ fixed is 

$$|\overline{\Gamma}_{HK}\backslash \Iso(H,K)|.$$

\end{cor}

\bigskip
If the associated subgroups $H$ and $K$ are conjugate in $G_1$, we can make then more explicit the formular of Corollary \ref{general iso}. 
Denote by $\widetilde{\aut}_{G_1}(H)$ and by $\widetilde{N}_{G}(H)$ the natural images of $\aut_{G_1}(H)$ and $N_G(H)$ in the group of outer automorphisms $\out(H)$ of $H$.

\begin{cor} Let $H$ and $K$ be conjugate finite subgroups of an $\OE$-group $G_1$. Then the number  of isomorphism classes of HNN-extensions \break $\hnn(G_1, H, K, f,t)$ 
 is $$|\widetilde{N}_{G_1}(H) \backslash \out(H)/  ( \widetilde{\aut}_{G_1}(H) \times C_2) |,$$
 where $C_2$ acts on $\out(H)$ by inversion. 
\end{cor}

Now similarly, denote by $P$ the set of all finite subgroups of $\widehat G_1.$ There is a natural action of $\widehat G_1\rtimes \aut_{\widehat G_1}(H)$ on the set $P$, namely \begin{equation} (g_1,\alpha)\cdot X= g_1\alpha(X)g_1^{-1} \in P, \mbox{ where } X \in P. \end{equation}
 We denote by $\Gamma_{\widehat{HK}}$  the stabilizer of $K\in P$ of this action. 
  Then $\Gamma_{\widehat{HK}}$ also acts naturally on the set of isomorphisms $\Iso(H,K)$. 
  
  We denote by $\widetilde\Gamma_{\widehat{HK}}$ the image of $\Gamma_{\widehat{HK}}$ in the symmetric group $Sym(\Iso(H,K)).$ Depending on the choice   $H$ and $K$  it might exist $\psi\in \aut(\widehat G_1)$  that swaps the conjugacy classes of $H$ and $K$ in $\widehat G_1$, say $\psi(H)=K$ and  $\psi(K)=H^{g_1}$ for some $g_1\in \widehat G_1$. In this case the "invertion" of elements of $\Iso(H,K)$ followed by the action by the element $( g', \psi)\in \widehat G_1\rtimes \aut(\widehat G_1)$ 
  gives an element $\iota\in Sym(\Iso(H,K))$ that normalizes $\widetilde\Gamma_{\widehat{HK}}$ 
  and so $\overline\Gamma_{\widehat{HK}}:=\widetilde\Gamma_{\widehat{HK}} \cdot \langle\iota\rangle$ is a group that does not depend in fact upon a choice of $\psi$ (see Lemma \ref{carmoprof}). If such $\psi$ does not exist we put $\overline\Gamma_{\widehat{HK}}:=\widetilde\Gamma_{\widehat{HK}}$.

In the profinite case we can only give an estimate for the number of isomorphism classes of profinite HNN-extensions with fixed associated subgroups, as follows.
\begin{thm}  Let $H, K$ be fixed finite subgroups of a profinite $\OE$-group $G_1$. Then the number of isomorphism classes of profinite HNN-extensions \\$\hnn(G_1, H, K, f, t)$ is less than or equal to $$|\overline{\Gamma}_{\widehat{HK}} \backslash \Iso(H, K)|.$$\end{thm}  



If the associated subgroups are conjugate in $\widehat G_1$  then in contrast to abstract case we get only an upper bound in terms of $H$.

\begin{thm} Let $H$ and $K$ be fixed finite subgroups that are conjugate in the profinite $\OE$-group $G_1$. Then the number of the isomorphism classes of profinite HNN-extensions $G=\hnn(G_1, H, K, f, t)$ is less or equal than $$\left| \widetilde N_{G_1}(H) \backslash \out(H)/ \left( \widetilde{\aut}_{G_1}(H) \times C_2\right) \right|,$$ where $C_2$ acts on $\out(H)$ by inversion. 
\end{thm}

After that, we study the genus of HNN-extensions where the base group $G_1$ is an $\rm \Oe$-group and the associated subgroups are finite. 


\medskip
Let $\mathcal{F}=\mathcal F(G_1,H,K)$ be the class of all 
abstract HNN-extensions $$G=\hnn(G_1, H, K, f, t),$$ where $H$ and $K$ are fixed finite subgroups of $G_1$. Denote by $N^+$  the following subset  
  \begin{equation*} 
N^+=\left\{ gt^{-1} \in N_{\widehat G}(H) \hspace{0.1cm} \left| \hspace{0.1cm}  \overline{\left\langle g, \widehat G_1 \right\rangle}=\widehat G \right. \right\}.   
\end{equation*} 
Let $\overline{N}^{+}$ and $\widetilde{N}^{+}$ be the natural images of $N^{+}$ in $\aut(H)$ and $\out(H)$, respectively. Given $G\in \F$ we first concentrate on calculation of the cardinality of genus $g(G,\F)$ (denoted by $|g(G, \F)|$)  within this family of HNN-extensions.

\begin{thm} Let $H$ and $K$ be finite subgroups of an $\Oe$-group $G_1$. Then the cardinality of the genus $g(G, \mathcal{F})$ of an HNN-extension $G=\hnn(G_1,H,K,f,t)$ is equal to $$|g(G, \mathcal{F})| = \left|\overline\Gamma_{HK}  \backslash \overline\Gamma_{\widehat{HK}}  \cdot  f \overline{N}^+\right|,$$ where $\overline\Gamma_{HK}  \backslash \overline\Gamma_{\widehat{HK}}  \cdot f\overline{N}^+$ represents the $\overline\Gamma_{HK}$-orbits in the set of   $\overline\Gamma_{\widehat{HK}}$-orbits of twists $f\tau_{n^{-1}}\,(n \in N^{+})$ of $f$. 
\end{thm}


An HNN-extension will be called normal if the stable letter normalizes the associated subgroup. Note that if $G$ is an HNN-extension with conjugate associated finite subgroups  in the base group then it is isomorphic to a normal HNN-extension (see Lemma \ref{conjugados da normal}), so in this case we can consider $f \in \aut(H)$. Let $\tilde f$ be the image of $f$ in $\out(H).$ It follows that $\tilde{f}\widetilde{N}^+$ is a subset of $\tilde f \widetilde{N}_{\widehat G}(H)$.

In this case  $\widetilde{\aut}_{\widehat G_1}(H)$ acts on $\out(H)$ on the right by conjugation. The group $\widetilde{N}_{\widehat G_1}(H)$ acts on $\out(H)$ on the left by multiplication. It follows that $\widetilde{N}_{\widehat G_1}(H) \cdot \tilde{f} \widetilde{N}^+ \cdot \widetilde{\aut}_{\widehat G_1}(H)$ is invariant with respect to the action of $\widetilde{N}_{G_1}(H)$ on the left and with respect to the action of $\widetilde{\aut}_{G_1}(H)$ on the right.

\begin{thm}  Let $G_1$ be an $\Oe$-group and $H$ a fixed finite subgroup of $G_1$. Let $G=\hnn(G_1,H, f,t)$ be a normal HNN-extension.  Then $$|g(G, \mathcal{F})| = | \widetilde{N}_{G_1}(H) \backslash \widetilde{N}_{\widehat G_1}(H) \cdot \tilde{f}\widetilde{N}^+ \cdot \widetilde{\aut}_{\widehat G_1}(H)/( \widetilde{\aut}_{G_1}(H) \times C_2)|,$$ where $C_2$ acts on $\out(H)$ by inversion. 
\end{thm}

We also establish a list of situations in which the genus formula is simpler (see Subsection \ref{section 5 1}), in particular, we find conditions for the genus cardinality to be  1 ( generalizing the Theorem 1.1 in \cite{BZ}). 

\begin{deff} A pair $(H,K)$ of subgroups of $G$ will be called a swapping pair if there exists an automorphism of $G$ that swaps the conjugacy classes of $H$ and $K$.  \end{deff}

\begin{thm}\label{genus 1 intr}
Let $H$ and $K$ be a finite subgroups of an $\Oe$-group $G_1$ and $G=\hnn(G_1, H, K, f, t)$ be an HNN-extension, where $f \in \Iso(H, K).$ 
Suppose the following conditions are satisfied \begin{itemize}
    \item $\widetilde{\aut}_{\widehat{G}_1}(H)=\widetilde{\aut}_{G_1}(H)$ and $[G_1:N_{G_1}(K)]=[\widehat G_1:N_{\widehat G_1}(K)]<\infty$,
    \item  $(H,K)$ is either a swapping pair in $G_1$ or not a swapping pair in $\widehat G_1$. 

\end{itemize} Then there is only one element in the genus of $G$ in $\F,$ if one of the following conditions is satisfied:

\begin{itemize}


\item[i)]  $\out(H)$  coincides with the image of the normalizer $N_{G_1}(H)$ in $\out(H)$;

\item[ii)] 
$H^{g_1}=K$ for some $g_1 \in G_1$ and $\tau_{g_1} f\in \overline{N}_{G_1}(H);$

    \item [iii)] $H$ and $K$ are not conjugate in $G_1$ and $f$ extends to an automorphism of $G_1;$  
    \item [iv)] $H$ and $K$ are not conjugate in $G_1$ and $N_{G_1}(K)=K \cdot C_{G_1}(K)$ or $N_{G_1}(H)=H \cdot C_{G_1}(H)$ (in particular, if $H \neq K$ and $H$ or $K$ is central in $G_1$).  
\end{itemize}
\end{thm}

The structure of the  paper is as follows.  Section \ref{preliminares} contains elements of the profinite version of the Bass-Serre theory used in the paper (see \cite{R} for more details) and some results on normalizers of profinite and  abstract HNN-extensions.  Futhermore it posseses  examples and properties of $\Oe$-groups. Sections \ref{iso abstract HNN} and \ref{profinitehnnextensions} dedicated to the isomorphism problem of (profinite) abstract HNN-extensions which are of independent interest. Moreover, we calculate the number of isomorphism classes of abstract and profinite HNN-extensions $G=\hnn(G_1, H, K, f, t)$ with not fixed associted subgroups $H,K$. More precisely,  the action of $\aut(G_1)$ on the set $S$ of finite subgroups of $G_1$ induces the action on the set of unordered pairs $[\overline{S}]^2$ of conjugacy classes of finite isomorphic subgroups of $G_1.$ 
The next theorem then gives a formula (resp. estimate) for the number of isomorphism classes of all abstract (resp. profinite) HNN-extensions of $G_1$ with finite associated subgroups.

\begin{thm} \label{int general finite associated} Let $G_1$ be  an $\OE$-group (resp. profinite $\OE$-group). Then the number  of isomorphism classes $n_a$ (resp. $n_p$) of  abstract (resp. profinite) HNN-extensions $\hnn(G_1, H_,K, f,t)$ with finite associated subgroups  is  $$n_a=\sum_{\{H, K\}} |\overline{\Gamma}_{HK} \backslash \Iso(H, K)|,$$
$$({\rm resp.}\  n_p\leq\sum_{\{H, K\}} |\overline{\Gamma}_{\widehat{HK}} \backslash \Iso(H, K)|) $$ where $\{H, K\}$ ranges over representatives of the $\aut(G_1)$-orbits of the set of unordered pairs $[\overline{S}]^2$ of conjugacy classes of finite isomorphic subgroups of $G_1$. 
\end{thm} 

The formula for such abstract (resp. profinite) HNN-extensions where associated subgroups are conjugate (equivalently coincide) is as follows:

\begin{thm}\label{int conjugate associated} Let $G_1$ be an $\OE$-group (resp. profinite $\OE$-group). Then the number of isomorphism classes of abstract (resp. profinite) HNN-extensions $\hnn(G_1, H, K, f, t)$ with finite conjugate associated subgroups  equals (resp. less or equal) $$\sum_{H} |\widetilde{N}_{G_1}(H) \backslash \out(H)/  ( \widetilde{\aut}_{G_1}(H) \times C_2) |,$$
 where $C_2$ acts on $\out(H)$ by inversion and $H$ ranges over representatives of the $\aut(G_1)$-orbits of finite subgroups of $G_1$.
\end{thm} 

In Section \ref{secao 5} we calculate the genus or estimate it for HNN-extensions. Sections 6 is dedicated to the calculation of genus and bounds in more general classes of groups.  Namely, Theorems \ref{int general finite associated} and \ref{int conjugate associated} combined with  \cite[Theorem 1.1]{BPZ}  can be used to compute  the genus of HNN-extension within a larger class $\A$ of finitely generated residually finite accessible groups whose vertex groups of its JSJ-decomposition are $\Oe$-groups  (see Theorem \ref{fixed H2} in Subsection  \ref{general section}).

{\subsection{Notations and terminologies} 

Our basic reference for notations and results about profinite groups is \cite{RZ}. We refer the reader to Lyndon-Schupp \cite{LS}, Magnus-Karrass-Solitar \cite{MKS}, Serre \cite{S} or Dicks-Dunwoody \cite{DI} for an account of basic facts on amalgamated free products  and to \cite{RZ} for the profinite versions of these constructions. Our methods based on the profinite version of the Bass-Serre theory of groups acting on trees that can be found in \cite{R}. 

We  give a brief description of the main notations and terminologies used throughout the paper. 

All homomorphisms of profinite groups are assumed to be continuous in this paper. The composition of  two maps $f$ and $g$ is often defined simply as $f\circ g = fg.$ We will use the standard abbreviation $x^{g}$ for $g^{-1}xg$ when $g,x$ are elements of a group $G;$ the internal automorphisms of $G,$ denoted by $\inn(G),$ is the group composed of the elements $\tau_g\, (g \in G)$ (as all maps are written on the left) such that $\tau_g(x)=gxg^{-1}, \forall x \in G.$ 
 
If $H$ is a subgroup of a group $G$, we will use $H^G$ for the normal closure of $H$ in $G,$  $N_G(H)$ for the normalizer of $H$ in $G$ and $C_G(H)$ for the centralizer of $H$ in $G.$  Furthermore, let $\aut_G(H)$ be the group of all automorphisms of $G$ that leave $H$ invariant and denote by $\overline{\aut}_G(H),$ $\widetilde{\aut}_G(H)$ its images in $\aut(H)$ and $\out(H),$ respectively. The natural image of $x \in N_G(H)$ in $\aut(H)$ will be given by $\tau_{x^{-1}},$ and let $\overline{N}_G(H)$ and $\widetilde{N}_G(H)$ denote the natural images of $N_G(H)$ in $\aut(H)$ and $\out(H),$ respectively.

 All abstract HNN-extensions are residually finite and in the profinite case, they will always be proper (i.e. the base group  embeds in the HNN-extension);  by Proposition 9.4.3 ii)-iii) in \cite{RZ} this is always the case if  $H, K$ are finite.

Throughout the paper we shall us the following families of groups. 
\begin{itemize}

\item $\mathcal{F}=\mathcal F(G_1,H,K)$ - be the class of all 
abstract HNN-extensions $$\hnn(G_1, H, K, f, t),$$ where $H$ and $K$ are fixed finite subgroups of $\Oe$-group $G_1$. 

\item $\A$ - the family of  finitely generated residually finite accessible groups whose vertex groups of its JSJ-decomposition are $\Oe$-groups. 

\item $\mathcal F(G, G_1, \A)$ - be the subclass of $\mathcal{A}$ consisting of all abstract HNN-extensions of the form $$B=\hnn(G_1, H_1, K_1, f_1, t_1),$$ where $G_1$ is an $\Oe$-group fixed  in the class $\mathcal{A}$ and the associated subgroups are finite.

\end{itemize}

\section{\large Preliminary Results} \label{preliminares}

In this section we will provide the basic and preliminary facts for the development of the main sections. 

\subsection{Basics of profinite trees} \label{basics of trees}

In this subsection we recall the necessary notions of the Bass-Serre theory for abstract and profinite graphs.
	
	\begin{deff}[Profinite graph]
		A (profinite) graph is a (profinite space) set $\Gamma$ with a distinguished closed nonempty subset $V(\Gamma)$ called the vertex set, $E(\Gamma)=\Gamma-V(\Gamma)$ the edge set and two (continuous) maps $d_0,d_1:\Gamma \rightarrow V(\Gamma)$ whose restrictions to $V(\Gamma)$ are the identity map $id_{V(\Gamma)}$. We refer to $d_0$ and $d_1$ as the incidence maps of the (profinite) graph $\Gamma$.  
	\end{deff}
	
	A morphism $\alpha:\Gamma \longrightarrow \Delta$ of profinite graphs is a continuous map with $\alpha d_i=d_i \alpha$ for $i=0,1$. By \cite[Proposition 2.1.4]{R} every profinite graph $\Gamma$ is an inverse limit of finite quotient graphs of $\Gamma$. A profinite graph $\Gamma$ is called connected if every its finite quotient graph is connected as a usual graph. 
	
	\begin{deff}\label{Profinite tree} Let $\Gamma$ be a profinite graph. Define $E^{*}(\Gamma)=\Gamma/V(\Gamma)$ to be the quotient space of $\Gamma$ (viewed as a profinite space) modulo the subspace of vertices $V(\Gamma)$. Consider the free profinite $\widehat\Z$-modules $[[\widehat\Z(E^{*}(\Gamma),*)]]$ and $[[\widehat\Z V(\Gamma)]]$ on the pointed profinite space $(E^{*}(\Gamma),*)$ and on the profinite space $V(\Gamma)$, respectively. Denote by $C(\Gamma,\widehat\Z)$ the chain complex

	$$	\xymatrix{ 0 \ar[r] & [[\widehat\Z(E^{*}(\Gamma),*)]] \ar[r]^d &[[\widehat\Z V(\Gamma)]] \ar[r]^{\varepsilon} & \widehat\Z \ar[r] & 0}$$ of free profinite $\widehat\Z$-modules and continuous $\widehat\Z$-homomorphisms $d$ and $\varepsilon$ determined by $\varepsilon(v)=1$, for every $v \in V(\Gamma)$, $d(\overline{e})=d_1(e)-d_0(e)$, where $\overline{e}$ is the image of an edge $e \in E(\Gamma)$ in the quotient space $E^{*}(\Gamma)$, and $d(*)=0$. 
		One says that $\Gamma$ is a profinite tree if the sequence $C(\Gamma,\widehat\Z)$ is exact. 
	\end{deff}	
	
	  If $v$ and $w$ are elements of a tree (respectively profinite tree) $T$, one denotes by $[v,w]$ the smallest subtree (respectively profinite subtree) of $T$ containing $v$ and $w$.

\medskip 


Let $G=\hnn(G_1, H, K, f, t)$ be a profinite HNN-extension. In contrast to the abstract case, $G_1$ do not always embed into $G$, i.e. a  profinite HNN-extension is not always proper. Note that $G$ is always proper  when $H$ and $K$ are finite (see Propositon $9.4.3$ item $(2)$ in \cite{RZ}) which will be assumed for starting from Subsection \ref{general results}.  In any case we shall consider only proper free profinite product with amalgamation and profinite HNN-extensions in this paper.

When $G=\hnn(G_1,H,K,f,t)$ is proper, there is
a corresponding    standard profinite tree  (or universal covering graph) 
  (cf. \cite[Theorem 3.8]{ZM1}) $S(G)=G/G_1\cup G/H$, $V(S(G))=G/G_1$ and  $d_0 (gH)= gG_1; \quad  d_1(gH)=gtG_1 $. There is a natural  continuous action of
 $G$ on $S(G)$, and clearly $ G\backslash S(G)$ is an edge with one vertex. 







\subsection{Normalizers}\label{normalizers}

The  description of   normalizers in this subsection is of independent interest for profinite groups.

\begin{prop} \label{normalizer}
\begin{enumerate}
\item[(i)] Let $G=G_1\amalg_H G_2$ be a proper  free profinite product with  amalgamation. Then $N_G(H)=N_{G_1}(H)\amalg_H N_{G_2}(H)$.
\item[(ii)] Let $G=\hnn(G_1, H, K, f, t)$ be a proper profinite HNN-extension with associated subgroups $H,$ $K$ and stable letter $t$. Then  
\begin{itemize}

\item[(a)]  If $H^{g_1}=K$ for some $g_1 \in G_1$ then $N_{G}(H)$ is a profinite HNN-extension  $N_{G}(H)=\hnn(N_{G_1}(H), H, \tau_{g_1} f)$. 

\item[(b)] If $H$ and $K$ are not conjugate subgroups in $G_1$ then $N_{G}(H)$ is a profinite amalgamated free product  $N_{G}(H)=N_{G_1}(H)\amalg_{H} N_{G_1^{t^{-1}}}(H)$. 
\end{itemize}

\end{enumerate}
\end{prop}

\begin{proof} 
Let $S(G)$ be the standard profinite tree on which $G$ acts continuously. Denote by $S(G)^{H}$ the subset of $S(G)$ of  fixed points with respect to the action of $H$. Then $S(G)^{H}$ is a profinite tree (Theorem 2.8 in \cite{ZM1} or Theorem 4.1.5 \cite{R}). Moreover $N_{G}(H)$ acts continuously on $S(G)^{H}$. Indeed, let $s \in S(G)^{H}, g \in N_{G}(H)$ and $h \in H,$ then: $$h\cdot(g\cdot s)=(hg)\cdot s=(gh')\cdot s = g\cdot (h'\cdot s)=g\cdot s,$$ for some $h'\in H,$  whence follows that $g\cdot s \in S(G)^{H},$ where $g\cdot s$ means the action of $g$ on $s$. 

Now the graph $S(G)^{H}/N_{G}(H)$ has only one edge. To see this it suffices to show that given $g\in G$ and $e\in E(S(G)^{H})$, then $ge\in E(S(G)^{H})$ implies  $g\in N_{G}(H)$. But this is obvious because both $H$ and $H^{g^{-1}}$ stabilize $ge$  which implies $H\leq  H^{g^{-1}}.$ However, in a profinite group, a conjugate of a subgroup cannot be contained in it properly; to see this, one need only look at the finite quotients and note that the conjugate subgroups have the same order. 
So $H = H^{g^{-1}}$ as needed.

 Thus $S(G)/G \cong S(G)^{H}/N_{G}(H)$ is an isomorphism of graphs and both are either a loop or an edge with two vertices. Note that the first case occurs when $G$ is an HNN-extension and $H$ and $K$ are conjugate in $G_1$ and the second case occurs when $H$ and $K$ are not conjugated in $G_1$ or $G$ is a free product with amalgamation. Now, note that $S(G)$ is connected and simply connected (see Theorem 3.4 in \cite{ZM} or Theorem 6.3.5 \cite{R}), thus  we can apply Proposition 4.4 in \cite{ZM} or Theorem 6.6.1 together with Proposition 6.5.3 in \cite{R} to deduce that $N_{G}(H)$ has the desired structures  of this proposition.
\end{proof}

In the following proposition gives a description of the normalizer of a  finite associated subgroup $H$ of an HNN-extension in the discrete and profinite categories.

\begin{prop}\label{cp1} 
Let $G = \hnn(G_{1}, H, K, f, t)$ be an HNN-extensions of a residually finite group $G_1$ with finite associated subgroups $H$ and $K$ and stable letter $t$. Then the following hold: \begin{itemize}

\item[(i)] If $H^{g_1}=K$ for some $g_1 \in G_1$ then $N_{G}(H)=\hnn(N_{G_1}(H), H, \tau_{g_1} f)$. 
If $H$ and $K$ are not conjugate subgroups in $G_1$ then $N_{G}(H)=N_{G_1}(H)*_{H} N_{{G_1}^{t^{-1}}}(H)$.

\item[(ii)] If $H^{g_1}=K$ for some $g_1 \in \widehat{G_1}$ then $N_{\widehat{G}}(H)$ is a profinite HNN-extension (i.e. the profinite completion of the usual HNN-extension) $N_{\widehat{G}}(H)=\widehat{\hnn}(N_{G_1}(H), H, \tau_{g_1} f)=\hnn(N_{\widehat{G_1}}(H), H, \tau_{g_1} f)$. If $H$ and $K$ are not conjugate subgroups in $\widehat{G_1}$ then $N_{\widehat{G}}(H)$ is a profinite amalgamated free product (i.e. the profinite completion of the usual amalgamated free product) $N_{\widehat{G}}(H)=N_{\widehat{G_1}}(H)\amalg_{H} N_{\widehat{G_1}^{t^{-1}}}(H)$. 
\end{itemize}
\end{prop}

\begin{proof} (ii)   Follows from Proposition \ref{normalizer}.  

\medskip
(i) The abstract version of the proof is analogous to the proof of  Proposition \ref{normalizer} and will be omitted;  the abstract version of results quoted in that proof  can be found in \cite{S} or \cite{DI}. 
\end{proof}


\begin{cor}\label{8} If $G_1$ is a polycyclic-by-finite group, then $N_{G}(H)$ is dense in $N_{\widehat{G}}(H)$. Furthermore,  $$\overline{N}_{\widehat{G}}(H)=\overline{N}_{G}(H)$$ where $\overline{N}_{G}(H)$ and $\overline{N}_{\widehat{G}}(H)$ are the natural images (by restriction) of $N_{G}(H)$ and $N_{\widehat{G}}(H)$ in the automorphism group $\aut(H),$ respectively. 
\end{cor}

\begin{proof} First observe that $H$ and $K$ are conjugate in $G_1$ if and only if they are conjugate in $\widehat G_1$  (see \cite[ Chapter 4, Theorem 7]{Se}).

Suppose that $H$ and $K$ are not conjugate in $\widehat{G_1}$. Then by Proposition \ref{cp1} we have   $$N_{G}(H)=N_{G_1}(H)*_{H} N_{G_{1}^{t^{-1}}}(H) \mbox{ and } N_{\widehat{G}}(H)=N_{\widehat{G_1}}(H)\amalg_{H} N_{\widehat{G_{1}}^{t^{-1}}}(H),$$ and so  $\left\langle N_{\widehat{G_1}}(H), N_{\widehat{G_{1}}^{t^{-1}}}(H)\right\rangle$ is dense in $N_{\widehat{G}}(H)$. As $G_1$ is polycyclic-by-finite it follows from Proposition 3.3 in \cite{RSZ} that $N_{G_{1}}(H)$ is dense in $N_{\widehat{G}_{1}}(H)$ and $N_{G_{1}^{t^{-1}}}(H)$ is dense in $N_{\widehat{G_{1}}^{t^{-1}}}(H)$, therefore  $N_{G}(H)$ is dense in $N_{\widehat{G}}(H)$. 

If $H$ and $K$ are conjugate subgroups in $\widehat{G_1}$, then by Proposition \ref{cp1} we have $$N_{G}(H)=\hnn(N_{G_1}(H), H, \tau_{g_1} f) \mbox{ and } N_{\widehat{G}}(H)=\hnn(N_{\widehat{G_1}}(H), H, \tau_{g_1} f).$$ In this case $N_{G}(H)=\left \langle N_{G_1}(H), tg_1^{-1}\right\rangle$ and $N_{\widehat{G}}(H)=\overline{\left \langle N_{\widehat{G_1}}(H), tg_1^{-1}\right\rangle}$, and the result follows similarly as in the case of the non-conjugate associated subgroups. 

\medskip
Since $H$ is finite we have 
  $$\overline{N}_{\widehat{G}}(H)=\overline{N}_{G}(H),$$ since 
$N_{G}(H)$ is dense in $N_{\widehat{G}}(H)$.


\end{proof}



\subsection{OE-groups} \label{OE groups}




A finitely generated residually finite group $G$ with less than 2 ends, i.e. a finite or one-ended group, will be called an $\OE$-group in the paper.   It follows from the famous Stallings theorem \cite{Sta} that $G$ is an $\OE$-group  if and only if whenever it acts on a tree with finite edge stabilizers
it  has a global fixed point. A finitely generated profinite group will be called $\OE$-group if it has the same property:   whenever it acts on a  profinite tree  with finite edge stabilizers, then it fixes a vertex. 

 A finitely generated residually finite group $G$  will be called $\Oe$-group if $\widehat G$ is $\OE$-group. Note that an $\Oe$-group is automatically $\OE$-group, since if a finitely generated residually finite group $G$ splits as an amalgamated free product or HNN-extension over a finite group then so does $\widehat G$.

Since we do not know whether 1-endedness is a profinite property, we do not know any single example of finitely generated residually finite 1-ended group which is not $\Oe$.  In this section we prove   
 several propositions that give  sufficient conditions for $G$ to be an $\Oe$-group  showing that the class of $\Oe$-groups is quite large and contains many important families of examples of $\Oe$-groups.

\begin{prop}\label{virtually cyclic}(\cite[Proposition 3.1]{BPZ}) Let $G$ be a finitely generated residually finite group such that $\widehat G$ does not have a non-abelian free pro-$p$ subgroup. If  $G$ is not virtually infinite cyclic, then $G$ is an $\Oe$-group.
\end{prop}


If $G$   satisfies an identity, then $\widehat G$ satisfies the same identity and so does not have non-abelian free  pro-$p$ subgroups, so it is an $\OE$ and an $\Oe$-group by Proposition \ref{virtually cyclic} unless it is virtually infinite cyclic. 

The result below was proven in Propositions 3.2 and 3.3 of \cite{BPZ2} and gives several examples of $\Oe$-groups.

\begin{prop}\label{examples} The following groups are $\Oe$-groups. \begin{enumerate}
\item[(i)]  Virtually surface groups;
\item[(ii)] Indecomposable (into a free product) 3-manifold groups;
\item[(iii)] Indecomposable (into a free product) finitely generated Right-Angled Artin groups;
 \item[(iv)] Arithmetic groups of $rank >1$;
 \item[(v)] Just infinite and finitely generated residually finite Fab groups. 
 \end{enumerate}
\end{prop}  

The following result shows that  the class $\OE$-groups is  closed for extensions and commensurability.

\begin{prop} \label{extensao de OE group}  \begin{enumerate}
    \item[(i)] Let $H$ be a normal (closed) subgroup of a (profinite) group $G.$ If $H$ and $G/H$ are (profinite) $\OE$-groups then $G$ is a (profinite) $\OE$-group; 
      
    \item[(ii)]  Let $H$ be a (closed) subgroup of finite index in a (profinite) group $G.$ If $H$ is an (profinite) $\OE$-group then $G$ is an (profinite) $\OE$-group.     
\end{enumerate}
\end{prop} 

\begin{proof} Let $T$ be a (profinite) tree with finite edge stabilizers on which $G$ acts (continuously) and without inversions.  

1. Note that $G/H$ acts on the tree $T^{H} \neq \emptyset.$ Therefore $G/H$ fixes some vertex of $T^{ H}$ and therefore $ G$ fixes a vertex of $T.$ 

 2.  Let $H_{G}$ be the core of $H$ in $G.$ It is clear that $H_{G}$ is a (closed) normal subgroup of finite index in $G$ (hence an open subgroup in the profinite case). Therefore $T^{H_{G}} \neq \emptyset$ and $G/ H_{G}$ acts on the tree $T^{H_{G}}.$ As $G/ H_{G}$ is finite, by \cite[Theorem 2.10]{ZM} $T^{H_{G}}$ has a fixed point for this action, therefore  $T^{G} \neq \emptyset.$       
\end{proof}

\subsection{ General results} \label{general results}

The following results will be important  for  next Sections. We shall prove the profinite version of the first statement only, as the abstract version follows mutadis mutandis.  

\begin{prop}\label{normal isomorfismo} Let $G = \hnn(G_{1}, H, f, t)$ and $B = \hnn(B_{1}, H_1, f_1, t_1)$ be abstract HNN-extensions with $H$ and $H_1$ finite associated subgroups of $G_1$ and $B_1$, respectively. 
Suppose $G_1$ is an $\OE$-group and $B$ is finitely generated  residually finite. If $G\cong B$ then there exists an isomorphism $\psi: G \longrightarrow B$ such that $\psi(G_1)= B_1$ and   $\psi(H)= H_1^{t_1^{\epsilon}}, \mbox{ where } \epsilon \in \{0,1\}.$  
\end{prop}

\begin{prop} \label{profinito normal iso} Let $G = \hnn(G_{1}, H,f,t)$ and $B = \hnn(B_{1}, H_1,f_1,t_1)$ be profinite HNN-extensions  with $H$ and $H_1$ finite associated subgroups of $G_1$ and $B_1$, respectively. Suppose $G_1$ is a profinite $\OE$-group and $B_1$ is finitely generated. If $G\cong B$ then there is a continuous isomorphism $\psi: G \longrightarrow B$ such that $\psi(G_1)= B_1$ and $\psi(H)= H_1^{t_1^{\epsilon}}$, where $ \epsilon \in \{0,1\}$.
\end{prop}

\begin{proof} Let $\psi: G\rightarrow  B$ be an isomorphism. Then $\psi(G_1)$ is conjugate into $B_1$, say $\psi(G_1)\leqslant {B_{1}}^{w}$ for some $w\in B$ (cf. \cite[Example 6.3.1]{R}). Thus replacing $\psi$ by $\tau_{w} \circ \psi$, if necessary (recall that $\tau_w$ is the inner  automorphism corresponding to $w$), we may assume that $\psi(G_1) \leqslant B_1$.  
Note that $H=G_1\cap G_1^{t^{-1}}$, so by Corollary 7.1.5 c) in \cite{R} we have $$\psi(H) = B_{1} \cap  B_{1}^{\psi(t^{-1})} \leqslant H_1^{t_1^{\epsilon}b}, \mbox{ where } b\in B_1 \mbox{ and } \epsilon \in \{0,1\}.$$ Thus replacing $\psi$ by $\tau_{b} \circ \psi$,  we get $\psi(H) \leq H_1^{t_1^{\epsilon}}.$ Suppose without loss of generality that $\psi(H) \leq H_1$. If the base groups of $G$ and $B$ are finite, applying the inverse homomorphism $\psi^{-1}$ in the previous calculations it can be shown that $\psi(G_1)=B_1$ and $\psi(H)=H_1$ because the groups will have the same orders, respectively. 

We now show that the containments $\psi(G_1) \leqslant B_1$ and $\psi(H) \leqslant H_1$ are in fact equalities even if the base subgroups are not finite. 
Since $\psi$ is an isomorphism $B=\hnn(\psi(G_{1}),  \psi(H), \ \psi \circ f\circ \psi^{-1}, \psi(t)  ).$ Let $U$ be an open normal subgroups of $B$ that intersect $H_1$  trivially. 
Set $U_1 := U \cap B_1$. Note that  $U_1$ is an open normal subgroup of $B_1$ satisfying $U_1 \cap H_1 =\{1\}=U_1 \cap \psi(H).$    
Let $\widetilde U= \overline{  U_{1}^{B}}$ the (topological) normal closure of subgroup $ U_1$ of $B.$ Consider the profinite HNN-extensions  $$\hnn( \psi(G_1)U_1/U_1,  \psi(H),  \psi \circ f\circ \psi^{-1}, \psi(t))\  {\rm and}\  \hnn( B_1/U_1,  H_1,  f_1, t_1)$$ of finite groups. We will show that such groups are isomorphic. Indeed, let $$\pi_1: B \rightarrow \hnn( B_1/U_1,  H_1,  f_1, t_1)$$ and $$\pi_2: B \rightarrow \hnn( \psi(G_1)U_1/U_1,  \psi(H),  \psi \circ f\circ \psi^{-1}, \psi(t))$$ be the natural epimorphisms. Note that in both cases the kernel is $\widetilde U$. Indeed to show that $Ker(\pi_2)=\widetilde U$ consider the following commutative diagram of epimorphisms
$$\xymatrix{B\ar[r]\ar[rd]& \hnn( \psi(G_1)U_1/U_1,  \psi(H),  \psi \circ f\circ \psi^{-1}, \psi(t))\ar[d]\\
& \hnn( B_1/U_1,  H_1,  f_1, t_1)}$$

Note that the restriction $$\hnn^{abs}( \psi(G_1)U_1/U_1,  \psi(H),  f^{\psi^{-1}}, \psi(t))\rightarrow \hnn^{abs}( B_1/U_1,  H_1, f_1, t_1)$$ of the vertical map is the natural map of abstract HNN-extensions and therefore is injective. Since these groups are virtually free, so is the vertical map of the diagram that coincides with the map of  their profinite completions. Thus the vertical map is an isomorphism. As the base groups are finite, it follows as previously discussed that $$(\psi(G_1)U_1)/U_1= B_1/U_1 \mbox{ and } H\cong H_1.$$ Since this holds for any $U$ such that $U\cap H_1=\{1\}$ we deduce that $\psi(H) = H_1$ and $\psi(G_1) =B_{1}.$ 
\end{proof}

The following result will be used later. It shows that  an HNN-extension cannot be isomorphic to an  amalgamated free product of $\OE$-groups.

\begin{prop} \label{HNN nao iso Amalg}
Let $G= \hnn(G_1, H_1, t_1, f_1)$ and $B=G_2 *_{H_2} G_3$ where $G_i \,(i=1,2,3)$ are $\OE$-groups and the $H_j \, (j=1,2)$ are finite. Then $G \not \cong B.$  
\end{prop}

\begin{proof} Suppose there exists an isomorphism $\psi: B \longrightarrow G.$ Since $G_i \, (i=2,3)$ are $\OE$-groups, we have that $\psi(G_i) \subset G_1^G \,(i=2,3),$ therefore $\psi$ is not surjective, an absurd. 
\end{proof} 


The result below shows that an HNN-extension with conjugate associated finite subgroups  in the base group cannot be isomorphic to an HNN-extension with non-conjugate associated finite  subgroups in its base group. 

\begin{prop} \label{normal non iso nao normal}
 Let $$G= \hnn(G_1, H, K, f, t)\  {\rm and}\  B=\hnn(B_1, H_1, K_1,f_1, t_1)$$ be HNN-extensions of an $\OE$-group $G_1$ with $H, K$ and $H_1, K_1$ finite subgroups of $G_1$ and $B_1$, respectively. Suppose $B$ is finitely generated residually finite. If $H$, $K$ are  conjugate in $G_1$ and $H_1$, $K_1$ are non-conjugated in $B_1$, then $G$ and $B$ are non-isomorphic. 
\end{prop} 

\begin{proof}
 Suppose on the contrary that there is an isomorphism $\psi: G \longrightarrow B.$ By Proposition \ref{normal isomorfismo} we have $\psi(G_1)=B_1$ and $\psi(H)=H_1 \mbox{ or } K_1.$  Let us assume without loss of generality that $\psi(H)=H_1.$ Then $$N_G(H) \cong \psi\left(N_{G}(H)\right)=N_{B}(H_1).$$ By Proposition \ref{cp1} \begin{equation} \label{cara dos normalizers} N_{G}(H)=\hnn(N_{G_1}(H), H, f') \cong N_{B}(H_1)=N_{B_1}(H_1)*_{H_1
 } N_{{B_1}^{t_{1}^{-1}}}(H_1).\end{equation}However, this is impossible since $N_G(H)\not \leq G_1^G$ and $N_B(H_1)\leq B_1^B$. 
\end{proof}


Using the same arguments as the previous proof and the Propositions \ref{cp1} 
we have the following profinite version: 

\begin{prop} \label{normal non iso nao normal profinite}
Let $G_1, B_1$ be profinite $\OE$-groups, $H, K$ and $H_1, K_1$ finite  subgroups of $G_1$ and $B_1$ respectively. Let $G= \hnn(G_1, H, K,  f,  t)$ and $B=\hnn(B_1, H_1, K_1, f_1, t_1)$ be profinite HNN-extensions such that in the first the associated subgroups are conjugate in $G_1$ and in the second they are non-conjugate in $B_1.$ Then $G$ and $B$ are non-isomorphic.  
\end{prop}


The following results will be important  for Sections \ref{iso abstract HNN} and \ref{secao 4}. We shall prove the profinite version of the  statement only, as the abstract version follows mutadis mutandis. 

\begin{prop}\label{fixed hnn abstract} Let $$G = \hnn(G_{1}, H,K, f, t)\  {\rm and}\  B = \hnn(B_{1}, H_1, K_1, f_1, t_1)$$ be abstract HNN-extensions, where $H, K$ and $H_1, K_1$ are finite associated subgroups of $G_1$ and $B_1,$ respectively.  Suppose that $G_1$ is an $\OE$-group and $B$ is residually finite. If $G\cong B$ then there is an isomorphism $\psi: G \longrightarrow B$ such that $\psi(G_1)= B_1$, $\psi(H)=H_1$ and $\psi(K)= K_{1}^{b_1}$, where $ b_1\in B_1$. Furthermore, if $f_1=\tau_{b_1}f^{\psi^{-1}}\in \Iso(H_1,K_1)$ then the converse also holds.
\end{prop}

\begin{prop}\label{2a} Let $$G = \hnn(G_{1}, H,K, f, t)\  {\rm and}\  B = \hnn(B_{1}, H_1, K_1, f_1,\\ t_1)$$ be profinite HNN-extensions, where $H, K$ and $H_1, K_1$ are finite associated subgroups.  Suppose that $G_1$ be a profinite  $\OE$-group.  If $G\cong B$ then there is a continuous isomorphism $\psi: G \longrightarrow B$ such that $\psi(G_1)= B_1$, $\psi(H)=H_1$ and $\psi(K)= K_{1}^{b_1}$, where $ b_1\in B_1$. Furthermore, if $f_1=\tau_{b_1}f^{\psi^{-1}}\in \Iso(H_1,K_1)$ then the converse of also holds.
\end{prop}

\begin{proof} Let $\psi: G\rightarrow  B$ be a  continuous isomorphism. By Proposition \ref{profinito normal iso} we can assume that $\psi(H) = H_1$ and $\psi(G_1) =B_{1}.$ 

\textbf{Case 1)}.  $H$ and $K$ are conjugate subgroups of $G_1$, i.e.,  $H^{g_1}=K$ with $g_1\in G_1$. Then by  Proposition \ref{normal non iso nao normal profinite}  $H_1$ and $K_1$ are conjugate subgroups of $B_1$, i.e.,  $K_1^{s}=H_1$ with $s\in B_1$. Therefore $\psi(K)=\psi(H^{g_1})=H_1^{\psi(g_1)}=K_1^{s\psi(g_1)}$. Put $b_1=s\psi(g_1) \in B_1$. Then $\psi(K)=K_1^{b_1}$ as needed.

\textbf{Caso 2)}.  $H$ and $K$ are not conjugate subgroups of $G_1$. Then by Proposition \ref{normal non iso nao normal profinite} $H_1$ and $K_1$ are not conjugates in $B_1$. It follows that
 $\psi(K) =\psi(G_1\cap G_1^{t})= B_1 \cap  B_1^{\psi(t)} =K_1^{b_1}$ for some $b_1\in B_1$  (we cannot have $\psi(K) = H_{1}^{ b_1}$ because  $H$ and $K$ are not conjugate in $G_1$).

\bigskip
Conversly, if $f_1=\tau_{b_1}f^{\psi^{-1}}$ we define a map $\varphi$  from $G$ to $B$ as follows $$\varphi: \left\{
\begin{array}{ccl}
g_1 & \longmapsto  & \psi(g_1), \forall  g_1 \in  G_1,\\
t & \longmapsto  & t_1 b_1\\
\end{array}
\right.$$ and observe that $\varphi$ is a continuous isomorphism.
\end{proof}

The result below is valid for abstract and profinite HNN-extensions.  

\begin{lem} \label{conjugados da normal} Let $G=\hnn(G_1, H, K, f, t)$ be an  abstract (profinite) HNN-extension such that $H^{g_1}=K$ for some $g_1 \in G_1.$ Then $$G \cong B=\hnn(G_1, H, H, f_1, t_1),$$  where $f_1=\left( \tau_{g_1} \circ f \right)^{\pm 1} \in \aut(H).$
\end{lem}

\begin{proof} First suppose that $f_1=\tau_{g_1} \circ f.$ Define the following isomorphism of $G$ to $\hnn(G_1, H, H, \tau_{g_1} \circ f, t_1),$  
 $$\psi: \left\{
\begin{array}{ccl}
g_1 & \longmapsto  & g_1, \forall  g_1 \in  G_1,\\
t & \longmapsto  & t_1 g_1.\\
\end{array}
\right.$$ Looking at presentation one sees indeed that this defines uniquely an isomorphism. To finish the proof, define the  isomorphism  $$\hnn(G_1, H, H, \tau_{g_1} \circ f, t_1)\longrightarrow \hnn(G_1, H, H, (\tau_{g_1} \circ f)^{-1}, t_{1}^{})$$   setting
 $$\varphi: \left\{
\begin{array}{ccl}
g_1 & \longmapsto  & g_1, \forall  g_1 \in  G_1,\\
t_1 & \longmapsto  & t_1^{-1}.\\
\end{array}
\right.$$ 
\end{proof}

\begin{prop}\label{nonnormalhnnabs} 
Let  $$G= \hnn(G_1, H, K, f, t)\  {\rm and}\  B=\hnn(G_1, H, K, f_1, t_1)$$ be HNN-extensions  with  associated finite subgroups in an $\OE$-group $G_1$. Suppose that $K=H$ or $H$ and $K$ are not conjugate in $G_1$. Then $G\cong B$ if and only if \begin{enumerate}

\item [(i)] there exist $\psi\in \aut(G_1)$ with $\psi(H)=K $ and $ \psi(K)=H^{g_1}$ for some $g_1 \in G_1$ such that $f_1^{\psi^{-1}}(k)= \tau_{g_1^{-1}} f^{- 1}(k), \forall k \in K,$ 
or 
\item [(ii)] there exist $\psi\in \aut(G_1)$ with $\psi(H)=H $ and $ \psi(K)=K^{g_1}$ for some $g_1 \in G_1$ such that $f_1^{\psi^{-1}}(h)= \tau_{g_1^{-1}} f(h), \forall h \in H.$ 
\end{enumerate}

\end{prop}

\begin{proof}
$(\Rightarrow)$ By Proposition \ref{fixed hnn abstract} there is an isomorphism $\psi: B \longrightarrow G$ such that $\psi(G_1) = G_1$,  $\psi(H)=H \mbox{ and }$ $\psi(K)=K^{w}$ or  $\psi(G_1) = G_1$, $\psi(H)=K \mbox{ and } \psi(K)=H^{w}$ for some $w\in G_1$.  

 If $H=K$ then $w\in N_{G_{1}}(H)$ and $\psi(t_1)=h_1t^{\pm 1}h_2\in N_{G}(H)$ with $h_1, h_2 \in N_{G_1}(H)$ by Lemma 2.11, p. 822 in \cite{BZ}.  As we have discussed throughout this article, we can replace the isomorphism $\psi$ by $\tau_{h_{1}^{-1}}\psi$, if necessary, to assume that $\psi(G_1)=G_1,$
  $\psi(H)=H$ and $\psi(t_1)= t^{\pm 1}h_2h_1.$ Put $g_1=h_2h_1 \in N_{G_1}(H)$. Then for  so for all $h \in H,$ we have $$\psi(f_1(h))=\psi(h^{t_1})=\psi(h)^{t^{\pm 1}g_1}=(\tau_{g_1^{-1}}f^{\pm 1}  \psi)(h).$$ 
 Therefore $f_1^{\psi^{-1}}(h) = \tau_{g_1^{-1}}f^{\pm 1}(h), \forall h \in H$ and (i) or (ii) holds. 

If $H$ and $K$ are not conjugate in $G_1$. We will divide the proof  into two parts.
  
 \textbf{ Case i)} Suppose that $\psi(G_1) = G_1$, $\psi(H)=K \mbox{ and } \psi(K)=H^{w}$ with $w\in G_1$.  Put $\psi(t_1) = g \in G$. Note that $$\psi(K) = H^{w}=\psi(H)^{\psi(t_1)}=K^g=H^{tg},$$ so that $w g^{-1} t^{-1} \in N_G(H)$. It follows from Lemma 2.11 in \cite{BZ} that $w g^{-1}=xty \in \left( N_{G_1}(H) \cdot t \cdot N_{G_1}(K) \right)$. Thus replacing $\psi$ by $\tau_y \circ \psi$, if necessary (where $\tau_y \in \overline{N}_{G_1}(K)$), we may assume that  $\psi(G_1)=G_1$, $\psi(H)=K$, $\psi(K)=H^{g_1}$ and $\psi(t_1)=t^{-1}g_1$, where $g_1=x^{-1}wy^{-1}.$ 
 For each $h \in H$, we have $$\psi(f_1(h))=\psi(h^{t_1})=\psi(h)^{\psi(t_1)}=\psi(h)^{t^{-1}g_1}=\tau_{g_1^{-1}} f^{-1}(\psi(h)).$$ 
 Therefore $f_1^{\psi^{-1}}(k)= \tau_{g_{1}^{-1}} f^{- 1}(k), \forall k \in K$ and (i) holds.

 \textbf{ Case ii)} Suppose $\psi(G_1) = G_1$,  $\psi(H)=H \mbox{ and }$ $\psi(K)=K^{w}$. Similar as in case (i), put $\psi(t_1) = g \in G$ and observe that $gw^{-1}t^{-1}\in N_{G}(H)$. It follows from Lemma 2.11 in \cite{BZ} that $g w^{-1}=xty \in \left( N_{G_1}(H) \cdot t \cdot N_{G_1}(K) \right)$. Thus replacing $\psi$ by $\tau_{x^{-1}} \circ \psi$, if necessary (where $\tau_{x^{-1}} \in \overline{N}_{G_1}(H)$), we may assume that $\psi(G_1)=G_1$, $\psi(H)=H$, $\psi(K)=K^{g_1}$ and $\psi(t_1)=tg_1$, where $g_1=ywx.$ 
 For each $h \in H$, we have$$\psi(f_1(h))=\psi(h^{t_1})=\psi(h)^{\psi(t_1)}=\psi(h)^{tg_1}=\tau_{g_1^{-1}} f(\psi(h)),$$
 Therefore $f_1^{\psi^{-1}}(h)= \tau_{g_1^{-1}} f(h), \forall h \in H$ and (ii) holds.

$(\Leftarrow)$ Follows from the Proposition \ref{fixed hnn abstract}. 

\end{proof}

\section{Isomorphism problem for abstract HNN-extensions} \label{iso abstract HNN}

In this section we will deal with the isomorphism problem of abstract HNN-extensions. 


 
Let $G=\hnn(G_1, H, K, f_1, t_1) \mbox{ and } B=\hnn(G_1, H, K, f_2, t_2)$ be abstract HHN-extensions, where $G_1$ is an OE-group and $f_{i}: H\longrightarrow K\,(i=1,2)$ are  isomorphisms between finite subgroups $H, K$ of $G_1$, i.e $f_i \in \Iso(H, K)\,(i=1,2)$. We find necessary and 
sufficient conditions for $G\cong B$ in terms of the orbits of the actions that will be defined below. Note that we are not trying to find an algorithm that decides whether given HNN-extensions are isomorphic. Our goal is to describe in terms of HNN-extensions when they are isomorphic. 
Some parts of the text in this section are adaptations of Sections 3 and 4 in \cite{BZ}.

Let $G=\hnn(G_1, H, K, f, t)$ be an abstract HNN-extension.  Consider the subgroup 
 \begin{equation} \label{gamma H} \Gamma_H:=G_1 \rtimes \aut_{G_1}(H)\end{equation} of the holomorph $G_1\rtimes \aut(G_1)$. 
Sometimes we shall interpret the action of an element $g_1$ of $G_1$ in $\Gamma_H$ as the inner automorphism $\tau_{g_1},$ which should not be confused with the element of $\aut(G_1)$ (from the copy of $\inn(G_1)$ in $\aut(G_1)$).

Denote by $S$ the set of all finite subgroups of $G_1$. There is a natural action of $\Gamma_{H}=G_1\rtimes \aut_{G_1}(H)$ on the set $S$ namely \begin{equation} \label{acaoemX} (g_1,\alpha)\cdot X= g_1\alpha(X)g_1^{-1} \in S, \mbox{ where } (g_1,\alpha) \in \Gamma_H, X \in S. \end{equation}

 Let us  consider now
$\Upsilon :=\displaystyle\bigcup_{X\in S} \Iso(H,X)$.  Then  the action of $\Gamma_H$ on $S$ induces the natural action
of  $\Gamma_H$  on the left on $\Upsilon$ as the following lemma shows.

\begin{lem}\label{elementary}
The semidirect product  $\Gamma_{H}=G_1 \rtimes \aut_{G_1}(H)$
acts from the left upon $\Upsilon$ by
\begin{equation} \label{acaosemidireta} (g_1,\alpha) \cdot f=\tau_{g_1} \alpha f \alpha^{-1} = \tau_{g_1} f^{\alpha^{-1}}, \mbox{ where } (g_1,\alpha) \in \Gamma_H \mbox{ 
 and } f \in \Upsilon.
 \end{equation}
\end{lem}
\begin{proof} For showing that the action is well-defined choose arbitrary $f\in \Upsilon$ and $\gamma=(g_1,\alpha)\in \Gamma_{H}$. Then $(g_1,\alpha) \cdot f \in \Iso(H, L)$ for $L:=\left[ \alpha(f(H))\right]^{g_1^{-1}} \in S$ and so $\gamma \cdot f \in \Upsilon$. Now certainly $(1,1)\cdot f=f$ holds for all $f\in \Upsilon$. Finally, if  $(g_1,\alpha), (g_2,\beta) \in \Gamma_H$ and $f \in \Upsilon$ then we have $$(g_1,\alpha) \cdot ((g_2,\beta) \cdot f)=\tau_{g_1}(\tau_{g_2} f^{\beta^{-1}})^{\alpha^{-1}}=\tau_{g_1}\alpha\tau_{g_2}\beta f \beta^{-1}\alpha^{-1}=$$ $$=\tau_{g_1}\tau_{g_2}^{\alpha^{-1}}f^{(\alpha\beta)^{-1}}=((g_1,\alpha)(g_2,\beta)) \cdot f.$$
\end{proof}



 Note that for $f\in \Iso(H, K)$, we have $(g_1,\alpha) \cdot f \in \Iso\left(H, [\alpha(K)]^{g_1^{-1}}\right) \subset \Upsilon$ holds true ( see (\ref{acaosemidireta})).  

Suppose there exist $\psi \in \aut(G_1)$ that swaps the conjugacy classes of $H$ and $K$, i.e. such that $\psi(K)=H^{g_1}$ and $\psi(H)=K$, for some $g_1\in G_1$.  Define $\iota: \Iso(H,K) \longrightarrow \Iso(H,K)$ by \begin{equation} \label{def iota abst} \iota(f)=\tau_{\psi^{-1}(g_1^{-1})}\left(f^{\psi}\right)^{-1}, f \in \Iso(H,K).\end{equation} 


When $H=K$ we can choose $\psi=id$ and $g_1=1,$ so $\iota$ sends $f$ to $f^{-1}.$ In this situation we have an isomorphism between $\hnn(G_1,H,f)$ and $\hnn(G_1,H,f^{-1})$ (see Lemma \ref{conjugados da normal}). Note that $\iota$ will be a bijection from $\Iso(H, H)=\aut(H)$ to $\Iso(H, H)=\aut(H)$ in this case. More generally, $\iota$ belongs to the symmetric group $Sym(\Iso(H, K)),$ of $\Iso(H, K)$.

We denote by $\Gamma_{HK}$ be the stabilizer of $K\in S$ of this action of $\Gamma_H$ on $S$. Then $\Gamma_{HK}$ leaves $\Iso(H, K)$ invariant according to the action of $\Gamma_H$ on $\Upsilon$.  
Since $\Gamma_{HK}$ acts on $\Iso(H, K),$ we can consider the natural image of $\Gamma_{HK}$ in $Sym(\Iso(H, K))$ and denote it by $\widetilde\Gamma_{HK}$.



\begin{lem} \label{semidiretocomC2} Suppose there exist $\psi \in \aut(G_1)$ that swaps the conjugacy classes of $H$ and $K$.  Then $\iota $ normalizes $\widetilde\Gamma_{HK}$ and so $\langle \widetilde\Gamma_{HK}, \iota  \rangle=\widetilde\Gamma_{HK} \cdot  \left\langle  \iota \right\rangle$ is a subgroup of $Sym(\Iso(H, K))$. \end{lem}  

\begin{proof}  We need to show that if $\tilde\omega\in \widetilde\Gamma_{HK}$ then $\tilde\omega^{\iota}\in \widetilde\Gamma_{HK}$. Indeed,  pick $\omega=(x, \alpha) \in \Gamma_{HK} $ such that $\tilde w$ is its image in $\widetilde\Gamma_{HK}.$  Choose $f\in \Iso(H,K)$ and observe that

$\tilde\omega^\iota. f=\iota^{-1}\tilde\omega(\tau_{\psi^{-1}(g_1^{-1})}(f^{-1})^\psi)= \iota^{-1}(\tau_{x\alpha\psi^{-1}(g_1^{-1})}(f^{-1})^{\psi\alpha^{-1}})=\newline =(f)^{\psi\alpha^{-1}\psi^{-1}}\tau_{\psi(\alpha\psi^{-1}(g_1)x^{-1})g_1^{-1}}=\tau_{\psi(\alpha\psi^{-1}(g_1)x^{-1})g_1^{-1}} (f)^{\psi\alpha^{-1}\psi^{-1}\tau_{\psi(\alpha\psi^{-1}(g_1)x^{-1})g_1^{-1}}}= \newline =(\tau_{\psi(\alpha\psi^{-1}(g_1)x^{-1})g_1^{-1}}, \tau_{g_1\psi(x\alpha\psi^{-1}(g_1^{-1}))}\psi\alpha\psi^{-1}). f$. But $$(\psi(\alpha\psi^{-1}(g_1)x^{-1})g_1^{-1}, \tau_{g_1\psi(x\alpha\psi^{-1}(g_1^{-1}))}\psi\alpha\psi^{-1}) \in \Gamma_{HK}$$ and its image in $Sym(\Iso(H, K))$ is $\tilde{w}^{\iota} \in \tilde{\Gamma}_{HK},$ as desired. 
\end{proof}




\begin{lem}\label{carmo} The subgroup $\widetilde\Gamma_{HK} \cdot  \left\langle  \iota \right\rangle$ of $Sym(\Iso(H, K))$ does not depend upon the choice of $\psi$.
\end{lem}

\begin{proof} Consider $\zeta \in \aut(G_1)$ such that $\zeta(K)=H^{r}$ and $\zeta(H)=K$, for some $r\in G_1$. Let $\Theta \in Sym(\Iso(H,K))$ such that  $\Theta(f):=\tau_{_{\zeta^{-1}(r^{-1})}}\left(f^{-1}\right)^{\zeta}$, where $f \in \Iso(H, K)$. Take $\upsilon=(\zeta^{-1}(r^{-1}  \cdot g_1), \zeta^{-1}\psi) \in \Gamma_{HK}$ such that $\tilde{\upsilon}$ is its image in $\widetilde\Gamma_{HK},$ and observe that $\Theta=\tilde{\upsilon} \iota \in \widetilde\Gamma_{HK} \cdot \left\langle \iota\right\rangle.$
\end{proof}

\begin{deff}
\label{novaacao} Given finite subgroups $H$ and $K$ of $G_1$ we set $$\overline{\Gamma}_{HK}:=\widetilde\Gamma_{HK}\cdot \left\langle \iota\right\rangle$$ if there exist $\psi \in \aut(G_1)$ such that $\psi(K)=H^{g_1}$ and $\psi(H)=K$, for some $g_1\in G_1$ and  put $$\overline{\Gamma}_{HK}:=\widetilde\Gamma_{HK}$$ otherwise. Further let $$\overline\Gamma_{HK}\backslash \Iso(H,K)$$ denote the set of orbits of $\Iso(H, K)$ modulo the action  of $\overline{\Gamma}_{HK}$.
\end{deff}




\begin{rem} \label{novaacao2} When $H=K$ then $\Iso(H,H)=\aut(H)$. In this case one can define   $\Gamma_{HH}=\Gamma_{H}=N_{G_1}(H) \rtimes\aut_{G_1}(H)$ as being the stabilizer of $H$ of  the action of $\Gamma_H$ on  $S$ (see (\ref{acaoemX})) and 
$\overline{\Gamma}_{HH}=\widetilde\Gamma_{H}\cdot C_2$, where $C_2= \left\langle \iota\right\rangle$ acts by inversion on $\aut(H)$. Note that if $\iota \in \widetilde\Gamma_H$ then $\overline{\Gamma}_{HH}=\widetilde\Gamma_{H},$ otherwise  $\overline{\Gamma}_{HH}=\widetilde\Gamma_{H}\rtimes C_2.$ 
In particular, $\overline{\Gamma}_{HH} \backslash \aut(H)$ denote the set of orbits of $\aut(H)$ modulo the action  of $\overline{\Gamma}_{HH}$.
\end{rem}



The result below is a consequence of  Propositions   \ref{nonnormalhnnabs} and the results of this subsection. They generalize Theorem 4.5 in \cite{BZ} to $\OE$-groups. 


\begin{thm} \label{classnonnormalabs} Let $G_1$ be an $\OE$-group and $$G=\hnn(G_1, H, K, f, t),\  B=\hnn(G_1, H, K, f_1, t_1)$$ be  abstract HNN-extensions with finite associated  subgroups. Then $G$ and $B$ are isomorphic iff $f$ and $f_1$ belong to the same $\overline\Gamma_{HK}$-orbit of $\Iso(H, K)$.    
\end{thm}

\begin{proof} 

$(\Rightarrow)$ We have two possibilities: either $H$ and $K$ are conjugate in the base group $G_1$ or  not.  Moreover, in the first case we may assume that $H=K$ by Lemma  \ref{conjugados da normal}.

\medskip
Thus suppose first that $H=K$. By Proposition \ref{nonnormalhnnabs} there  exist $\varphi\in \aut_{G_1}(H)$ and $ g_1 \in N_{G_1}(H)$ such that

$$( f_1)^{\varphi^{-1}}= \tau_{g_1^{-1}} ( f)^{\epsilon} \iff \tau_{g_1}f_1^{\varphi^{-1}}=( f)^{\epsilon},$$
where $\epsilon \in \{\pm 1\}.$ 
If $\epsilon=1$ take $(g_1,\varphi)\in \Gamma_{H}$ and  observe that $$(g_1,\varphi)\cdot f_1=\tau_{g_1}f_1^{\varphi^{-1}}= f.$$

If $\epsilon=-1$ take $\beta=(\varphi^{-1}(g_1^{-1}), \varphi^{-1})\in \Gamma_{H}$ and observe that  $\iota(f)=f^{-1}$ (just consider $\psi=id$ in the Definition of $\iota$ in (\ref{def iota abst})). In this way, we have 
$$(\beta, \iota) \cdot f = \tau_{\varphi^{-1}(g_{1}^{-1})} \left( f^{-1}\right)^{\varphi}   =f_1.$$ 
Therefore for any $\epsilon=\pm 1$ we have that $f_1$ and $f$ are in the same $\overline\Gamma_{HH}$-orbit.

\medskip

Now suppose
that $H$ and $K$ are not conjugate in $G_1$. By Proposition \ref{nonnormalhnnabs} $G \cong B$ if, and only if, the subcases i) or ii) of Proposition \ref{nonnormalhnnabs} are satisfied. 

 \textbf{Subcase i)}: There exist $\psi\in \aut(G_1)$ with $\psi(H)=K $ and $ \psi(K)=H^{g_1}$ for some $g_1 \in G_1$ such that $f_1^{\psi^{-1}}(k)= \tau_{g_1^{-1}} f^{- 1}(k), \forall k \in K$. By definition of  $\iota$ in (\ref{def iota abst}) we have

$$\iota(f)=\tau_{\psi^{-1}(g_1^{-1})}(f^{-1})^{\psi}=f_1.$$
Therefore $f_1$ and $f$ belong to the same  $\overline{\Gamma}_{HK}$-orbit of $\Iso(H, K)$.

 \textbf{Subcase ii)}: There exist $\psi\in \aut(G_1)$ with $\psi(H)=H $ and $ \psi(K)=K^{g_1}$ for some $g_1 \in G_1$ such that $f_1^{\psi^{-1}}(h)= \tau_{g_1^{-1}} f(h), \forall h \in H$. Observe that $(g_1, \psi)\in \Gamma_{HK}$ and $$(g_1, \psi)\cdot f_1= \tau_{g_1}f_1^{\psi^{-1}}=f.$$
Therefore $f_1$ and $f$ belong to the same orbit.

\bigskip

\hspace{-0,8cm} $(\Leftarrow)$ It follows from Proposition   \ref{nonnormalhnnabs}. 
\end{proof}

\begin{cor}\label{number of isomorphisms} Let $G_1$ be an $\OE$-group and $H\cong K$ be fixed finite subgroups of $G_1$.
The number of isomorphism classes of abstract HNN-extensions of the form $\hnn(G_1, H, K, f, t)$  equals to the number of $\overline\Gamma_{HK}$-orbits 

\begin{equation}\label{formula numero hnn extensao}
|\overline{\Gamma}_{HK} \backslash \Iso(H, K)|
\end{equation}
in $\Iso(H, K)$.
\end{cor}


\subsection{The number of isomorphism classes of abstract HNN-extensions 
} \label{cota bonita non conjugados}

 Our first goal is to count the isomorphism classes of HNN-extensions  $$\hnn(G_1, H, K, f, t)$$ with finite associated subgroups. 


Let $ S$ be the set of all finite subgroups of $G_1$. We have an action of $G_1$ on $S$ by conjugation. So, given $L\in S$ the orbit of $L$ with respect to this action (i.e., the conjugacy class of $L$ in $G_1$) is  
$$\langle\langle L\rangle\rangle^{G_1}= \{L^{g_1}| g_1 \in G_1\}.$$ 
Denote by $\overline{S}$ the space of orbits of this action (the set of conjugacy classes of finite subgroups of $G_1$ ), i.e., $$\overline{S}= G_1\backslash S:=\{ \langle\langle L\rangle\rangle^{G_1} \,|\, L \in S\}.$$    

Let $[\overline{S}]^2$ be the collection of all unordered pairs of the form $\{X, Y\}$ such that $X, Y \in \overline{S}.$  Define an action of $\aut(G_1)$ on $[\overline{S}]^2$ by \begin{equation*} \varepsilon: \aut(G_1) \times [\overline{S}]^2 \longrightarrow [\overline{S}]^2, \end{equation*}
\begin{equation} \label{acao epsilon}
( \varphi, \{ \langle\langle H_1\rangle\rangle^{G_1} , \langle\langle K_1\rangle\rangle^{G_1} \} ) \mapsto \{ \langle\langle \varphi(H_1)\rangle\rangle^{G_1} , \langle\langle  \varphi(K_1)\rangle\rangle^{G_1}  \}.
\end{equation}


Note that $\{ \langle\langle H\rangle\rangle^{G_1}
, \langle\langle K\rangle\rangle^{G_1}
 \}$ and $\{ \langle\langle H_1\rangle\rangle^{G_1}
, \langle\langle K_1\rangle\rangle^{G_1}
 \}$ belong to the same $\aut(G_1)$-orbit of action (\ref{acao epsilon}) if and only if  there exist $\psi\in \aut(G_1)$ such that $\psi(H)=H_1$ and $\psi(K)=K_1^{g_1}$ for some $g_1\in G_1.$ 

Let $H_1, K_1 \in S,$ by Proposition  \ref{fixed hnn abstract},   for any $f\in \Iso(H,K)$ one has  $$\hnn(G_1, H, K, f) \cong \hnn(G_1, H_1, K_1, \tau_{g_1}f^{\psi^{-1}}).$$

If $H$ and $K$ are finite subgroups of  $G_1,$ we saw in Corollary \ref{number of isomorphisms} that the number of isomorphism classes of abstract HNN-extensions of the form $\hnn(G_1, H, K, f, t),$ with $f \in \Iso(H, K)$ is 
$|\overline{\Gamma}_{HK} \backslash \Iso(H, K)|$. 

\medskip
Denote the set of orbits of action (\ref{acao epsilon}) as $\aut(G_1)\backslash [\overline{S}]^2$. From the preceding discussion and Theorem \ref{classnonnormalabs} we have  

\begin{thm} \label{general finite associated} Let $G_1$ be  an $\OE$-group. Then the number  of isomorphism classes of abstract HNN-extensions $$\hnn(G_1, H,K, f,t)$$ with finite associated subgroups  is $$\sum_{\{H, K\}} |\overline{\Gamma}_{HK} \backslash \Iso(H, K)|,$$ where $\{H, K\}$ ranges over representatives of the orbits $\aut(G_1)\backslash [\overline{S}]^2.$ \end{thm}

Let now $H$ be a finite subgroup of an OE-group $G_1.$ Let $S_{_{H}} \subset S$ be the set of all finite subgroups of $G_1$ isomorphic to $H$.
By Proposition \ref{fixed hnn abstract}   isomorphic HNN-extensions with finite associated subgroups must have isomorphic associated subgroups and by Proposition \ref{2a} isomorphic profinite HNN-extensions with finite associated subgroups must have isomorphic associated subgroups as well. 
Therefore for the calculation of genus it is important to count isomorphism classes of HNN-extensions with fixed associated subgroup $H$ up to isomorphism.

To do this we restrict  the action given by (\ref{acao epsilon}) of $\aut(G_1)$ on $[\overline{S}]^2,$ to the set $[\overline{S_H}]^2$ of all unordered pairs of conjugacy classes of finite  subgroups of $G_1$ isomorphic to $H$. Denote by $\aut(G_1) \backslash [\overline{S_H}]^2$ the space of orbits of this action. Thus we have the following

\begin{thm} \label{classe geral familia F} Let $H$ and $K$ be finite subgroups of an $\OE$-group $G_1$. Then the number $\eta$ of isomorphism classes of abstract HNN-extensions $$\hnn(G_1, H_1,\\
K_1, f_1,t_1)$$ such that 
$H_1 \cong H \cong K \cong  K_1$ and $f_1 \in \Iso(H_1, K_1)$ is $$\eta  = \sum_{\{H, K\}} |\overline{\Gamma}_{HK} \backslash \Iso(H, K)|,$$ where $\{H, K\}$ ranges over representatives of the orbits $\aut(G_1)\backslash [\overline{S}_H]^2.$
\end{thm} 


\subsection{Normal HNN-extension} \label{cota bonita conjugados}


\subsubsection{Fixed associated subgroups}

The formula (\ref{formula numero hnn extensao}) is a consequence of  Theorem \ref{classnonnormalabs} and calculates the number of isomorphism classes of HNN-extensions of the form $\hnn(G_1,H,K,f,t)$ with $H,K$ fixed finite subgroups of $G_1$, where $ f\in \Iso(H,K)$.
The objective of this subsection is to make more explicit  formula (\ref{formula numero hnn extensao})  by showing a more elegant version in the case where $H$ and $K$ are conjugate finite subgroups in $G_1$.
 For this, we observe that if $G=\hnn(G_1,H,K,f,t)$ where $H^{g_1}=K$, for some $ g_1\in G_1$ then $G\cong \hnn(G_1, H, (\tau_{g_1}f)^{\pm 1}, t), \mbox{ where } (\tau_{g_1}f)^{\pm 1} \in \aut(H) \mbox{ and } g_1 \in G_1$ (see Lemma \ref{conjugados da normal}).  

\begin{deff}
An (profinite) abstract HNN-extension in which the stable letter normalizers associated subgroup will be called \textbf{normal}. 
\end{deff}

For normal HNN-extensions  it is possible to write the set of orbits of the Theorem \ref{classnonnormalabs} in a more natural way and this will be extremely useful to deduce several corollaries of  of this theorem.

From Proposition \ref{nonnormalhnnabs} we have that  $\overline{N}_{G_1}(H)$ acts on $\aut(H)$ by composition on the left: we will denote by $\overline{N}_{G_1}(H) \cdot f =\{\tau_g \circ f \,|\, g \in N_{G_1}(H), f \in \aut(H) \}$  the orbit of $f$ with respect to this action and the space of right cosets $\overline{N}_{G_1}(H) \backslash \aut(H)$ is the space of these orbits. Note that $\overline{N}_{G_1}(H) \triangleleft \overline{\aut}_{G_1}(H)$; furthermore, $ \overline{\aut}_{G_1}(H) \times C_2$ acts on the right on $\aut(H)$, where $C_2=\left\langle \iota\right\rangle$ acts by inversion, as follows 
\begin{equation}\label{acao1} f \cdot (\psi, \iota) = \psi^{-1} f^{-1}\psi = (f^{-1})^{\psi}=\left(f^{\psi}\right)^{-1},\end{equation}
where  $( \psi, \iota) \in  \overline{\aut}_{G_1}(H) \times C_2, f \in \aut(H).$ This induces a natural action on the right of the group $\overline{\aut}_{G_1}(H) \times C_2$ on the space of orbits $\overline{N}_{G_1}(H) \backslash \aut(H)$, given by
\begin{equation} \label{acao2} \left( \overline{N}_{G_1}(H) \cdot f \right) \cdot (\psi, \iota) = \overline{N}_{G_1}(H) \cdot  (f^{-1})^{\psi}=\overline{N}_{G_1}(H) \cdot \left(f^{\psi}\right)^{-1},\end{equation}
where $( \psi, \iota) \in  \overline{\aut}_{G_1}(H) \times C_2, f \in \aut(H).$ We will denote the space of orbits of the action given in (\ref{acao2}) as $\overline{N}_{G_1}(H) \backslash \aut(H)/ ( \overline{\aut}_{G_1}(H) \times C_2)$; this is not the space of double cosets as $C_2$ is not a subgroup of $\aut(H)$ and the action of  $\overline{\aut}_{G_1}(H)$ is by conjugation.

From  Formulas (\ref{acao1}) and (\ref{acao2}) it can be seen that the set of orbits of the action first on the left by $\overline{N}_{G_1}(H)$ in $\aut(H)$ followed by the action of $ \overline{\aut}_{G_1}(H) \times C_2$ on $\aut(H)$ on the right  is in bijection with $\overline{\Gamma}_{HH} \backslash \aut(H)$, so from now on we will identify $\overline{\Gamma}_{HH} \backslash \aut(H) = \overline{N}_{G_1}(H) \backslash \aut(H)/ ( \overline{\aut}_{G_1}(H) \times C_2)$. 

Thus by 
Corollary \ref{number of isomorphisms} we have the following

\begin{thm} Let $G_1$ be an $\OE$-group and $H$ be a fixed finite subgroup of $G_1$.
Then there is a bijection between   the set of orbits \begin{equation} \label{a bijecao das orbitas} \overline{N}_{G_1}(H) \backslash \aut(H)/ ( \overline{\aut}_{G_1}(H) \times C_2) \end{equation} and the set of isomorphism classes of HNN-extensions of the form $\hnn(G_1, H, f, t),$ where $f\in \aut(H)$.\end{thm}

 Note that $\inn(H)\leq \overline{N}_{G_1}(H)$; so we can pass to $\out(H)$ using tilde instead of overline. 
Therefore
we can state the following

 \begin{cor}\label{corhnn} Let $H$ and $K$ be fixed conjugate finite subgroups of an $\OE$-group $G_1$. Then the number $n_{_{H}}$  
 of isomorphism classes of HNN-extensions $\hnn(G_1, H, K, f,t)$,  $f \in \Iso(H, K)$  
 is $$n_{_{H}}=|\widetilde{N}_{G_1}(H) \backslash \out(H)/  ( \widetilde{\aut}_{G_1}(H) \times C_2) |.$$ 
 \end{cor}  



\begin{cor}\label{g1} 
Suppose that $\widetilde{\aut}_{G_1}(H)$ is central in $\out(H)$. Then $\widetilde{N}_{G_1}(H) \\ \unlhd \out(H)$ and the number of isomorphism classes of groups of the form $G=\hnn(G_1,H,f, t)$, with $f\in \aut(H)$ is equal to $\frac{n+d}{2}$, where $d$ is the number of elements of order $\leq 2$ in the quotient  $\out(H)/\widetilde{N}_{G_1}(H)$ and $n=|\out(H)/\widetilde{N}_{G_1}(H)|$.
\end{cor}

\begin{proof} Since $\widetilde{N}_{G_1}(H)\leq \widetilde{\aut}_{G_1}(H)$ and $\widetilde{\aut}_{G_1}(H)$ is central in $\out(H)$, we have $\widetilde{N}_{G_1}(H)^{\gamma}=\widetilde{N}_{G_1}(H)$ for all $\gamma \in \out(H)$, i.e $\widetilde{N}_{G_1}(H)\unlhd \out(H)$.   Denote by $\tilde{f_1}$ the image of $f_1 \in \aut(H)$ in $\out(H).$ 
By formula (\ref{a bijecao das orbitas}) combined with Remark \ref{novaacao2} and Corollary \ref{corhnn}, the isomorphism class $[B]$ of $B=\hnn(G_1, H, f_1), \mbox{ with } f_1\in \aut(H)$ corresponds to $$\left\{ \tilde{\tau}_{k}  \left(\tilde{f}_1\right)^{\pm 1} \ | \ k
\in  N_{G_1}(H) \right\}=[f_1],$$ where $[f_1]$ denotes the orbit of $f_1.$ Note that we can partition $\widetilde{N}_{G_1}(H) \backslash \out(H)$ as $\tilde{N}_{1} \dot\cup 
\tilde{N}_{2},$ where $\tilde{N}_{1}$ is the set containing the orbits $\widetilde{N}_{G_1}(H) \cdot \tilde{f}_1$ such that $|\tilde{f}_1|> 2$ and $\tilde{N}_{2}$ is the set containing the orbits $\widetilde{N}_{G_1}(H) \cdot \tilde{f}_1$ such that $|\tilde{f}_1| \leq 2.$ Thus $$ (\out(H)/\widetilde{N}_{G_1}(H)) /  ( \widetilde{\aut}_{G_1}(H) \times C_2) =  \tilde{N}_1 /  ( \widetilde{\aut}_{G_1}(H) \times C_2) \dot \cup \tilde{N}_2 /  ( \widetilde{\aut}_{G_1}(H) \times C_2).$$ By (\ref{acao2}) and (\ref{a bijecao das orbitas}) we have that $|\tilde{N}_2 / ( \widetilde{\aut}_{G_1}(H) \times C_2)|=d$ and $|\tilde{N}_{1} / ( \widetilde{\aut}_{G_1}(H) \times C_2)|= \frac{n - d}{2}$ and so all together the number of orbits is $\frac{n - d}{2}+d=\frac{n + d}{2}$. 
\end{proof}


\begin{rem}
We note that Corollary \ref{g1} applies when $H$ is a finite cyclic subgroup in an $\OE$-group $G_1$. 
\end{rem}

Recall that a subgroup $H$ of $G$ is {\it malnormal} if $H\cap H^g=\{1\}$ for every $g \in G \backslash H$.


\begin{cor} \label{malnormal} 
If $\widetilde{\aut}_{G_1}(H)$ is central in $\out(H)$ and $\widetilde{N}_{G_1}(H)=1$ (in particular if $H$ is malnormal in $G_1$), then the number  of isomorphism classes of groups of the form $G = \hnn(G_1, H, f, t),$ with
$f \in \aut(H)$ is equal to $\frac{|\out(H)| +d}{2},$ where $d$ is the number of elements of order $\leq 2$ in $\out(H).$ 
\end{cor}

\subsubsection{Not fixed associated subgroups}

Let $L\leq G_1$ be a finite subgroup of $G_1$ isomorphic to $H$, i.e. $L \in S_H.$ Then by  Corollary \ref{corhnn} the number of isomorphism classes of abstract normal HNN-extensions  $\hnn(G_1, L, f, t)$ is  $$n_{_{L}}:=| \widetilde{N}_{G_1}(L) \backslash \out(L)/ ( \widetilde{\aut}_{G_1}(L) \times C_2) |.$$ Note that in our case $K$ belongs always to the same orbit as $H$ and so the action that we consider is $\aut(G_1)$  on $\overline S_H$ or simply on $S_H$. From Theorem \ref{classe geral familia F} then, we deduce then the following  



\begin{thm} Let $H$  be  a finite subgroup of the $\OE$-group $G_1$. Then the number of isomorphism classes of abstract HNN-extensions $\hnn(G_1, L, K, f, t)$ with conjugate associated subgroups isomorphic to $H$ equals $$\sum_{L} n_{_{L}},$$ where $L$ ranges over representatives of the orbits  $\aut(G_1)\backslash S_H$.
\end{thm}

\begin{cor} The number of isomorphism classes of abstract HNN-extensions $\hnn(G_1, L, K, f, t)$ with conjugate $L$ and $K$   in $G_1$ and isomorphic to $H$ is less or  equal to $$|\aut(G_1)\backslash {S}_H| \cdot |\out(H)/C_2|.$$
\end{cor}

 \begin{rem} If $|\out(H)|$ is odd then  $$| \out(H)/ C_2|=\frac{|\out(H)|-1}{2} + 1.$$ Therefore, in this case the number of isomorphism classes in the previous Corollary is less than or equal to $$|\aut(G_1)\backslash S_H| \cdot \left( \frac{|\out(H)|-1}{2} + 1\right).$$
\end{rem}



Now we deduce the formula for the number of isomorphism classes of all HNN-extensions with finite conjugate associated subgroups.

\begin{thm} Let $G_1$ be an $\OE$-group. Then the number of isomorphism classes of abstract HNN-extensions $\hnn(G_1, L, K, f, t)$ with finite conjugate associated subgroups  equals $$\sum_{L} n_{_{L}},$$ where $L$ ranges over representatives of the orbits  $\aut(G_1)\backslash S$.
\end{thm}

\begin{cor} The number of isomorphism classes of abstract HNN-extensions $\hnn(G_1, L, K, f, t)$ with $L$ and $K$ finite and conjugate  in $G_1$  is less or  equal to $$|\aut(G_1)\backslash {S}| \cdot |\out(H)/C_2|.$$
\end{cor}

\subsubsection{Examples}

\begin{exa} Let $G= G_1*_{H, \mu} (H\rtimes_{\alpha} \Z)$ be a free product with finite amalgamation, where $G_1$ is an $\OE$-group and  $\mu$ is a monomorphism from $H$ to $H\rtimes_{\alpha} \Z= H\rtimes_{\alpha} \left\langle t \right\rangle.$ Note that $G\cong \hnn(G_1, H, \gamma, t),$ where $\gamma = \alpha(t)\circ \mu \in \aut(H).$ This case was missing in the study of free products with finite amalgamation in \cite{BPZ2}. The results of this section fill this gap, in fact, the number of  isomorphism classes of such free products with amalgamation $ G_1*_{H, \mu} (H\rtimes_{\alpha} \Z)$  with $H$ fixed is 
$$n=| \widetilde{N}_{G_1}(H) \backslash \out(H)/ ( \widetilde{\aut}_{G_1}(H) \times C_2) |.$$
Moreover, if $H$ is malnormal, then 
$$n=| \out(H)/ ( \widetilde{\aut}_{G_1}(H) \times C_2) |.$$
\end{exa}

The following result is an adaptation of \cite[Example 4.8]{ BZ}.

\begin{exa}\label{bb}
Let $G_1$ be an $\OE$-group, $H\leq G_1$ such that $H$ is central cyclic of order $n$ in $G_1$. As $H$ is central in $G_1$ it follows that $\overline{N}_{G_1}(H)=1=\widetilde{N}_{G_1}(H)$. Note that 
$\aut(H)=\prod_{p| n} \aut(C_{p^m})=\prod_{p| n} \out(C_{p^m})=\out(H),$   where $p^m$ is the order of the $p$-Sylow subgroup of $H$. 
Note that
$$\aut(C_{p^m})=\left\{\begin{array}{l}
1, \mbox{  if } m=1 \mbox{ and } p=2,\\
C_{2}, \mbox{  if } m=2 \mbox{ and } p=2,\\
C_{2}\times C_{2^{m-2}}, \mbox{  if } m\geq3\mbox{ and } p=2,\\
C_{p-1}\times C_{p^{m-1}}, \mbox{  if } p \mbox{ odd.}
\end{array} \right.$$
Thus $\aut(H)$ is abelian, and moreover  $\aut(C_{p^m})$ contains $2$ elements of order $\leq 2$ if $p$ is odd and $4$ elements if  $p=2$ and $m\geq3$. Let $r \geq 0$ be the number of primes $p \neq 2$ dividing $n.$ Then by Corollary \ref{g1}  the number of isomorphism classes of the groups $G=\hnn(G_1, H, f, t)$, where $f \in \aut(H)$ equals to:
$$\left\{\begin{array}{l} 

\frac{\prod_{2\neq p|n} \left( (p-1)p^{m-1}\right) -2^r}{2} +2^r, \mbox{ if the 2-Sylow subgroup of } H \mbox { has order } \leq 2;\\
\\

\frac{2 \cdot  \prod_{2\neq p|n} \left( (p-1)p^{m-1}\right) -2^{r+1}}{2} +2^{r+1}, \mbox{ if the 2-Sylow subgroup of } H \mbox{ has  order } 4;\\
\\
\frac{2^{m-1} \cdot  \prod_{2\neq p|n} \left( (p-1)p^{m-1}\right) -2^{r+2}}{2} +2^{r+2}, \mbox{ if the 2-Sylow subgroup of } H \mbox{ has order } 2^m \\ \mbox{ with } m\geq 3.
\end{array} \right.$$

In particular, if $H\cong C_{p^m}$, then the number of isomorphism classes is equal to 
$$\left\{\begin{array}{l}
1, \mbox{  if } m=1 \mbox{ and } p=2,\\
2, \mbox{  if } m=2 \mbox{ and } p=2,\\
(2^{m-1}-4)/2 +4, \mbox{  if  }m\geq3 \mbox{ and } p=2, \\
\frac{(p-1)p^{m-1}-2}{2} +2,\mbox{ if } p \mbox{ odd}.
\end{array} \right.$$
\end{exa}


\begin{rem} \label{finite-by-cyclic} 
 
Let $G=\hnn(G_1, H, K, f, t)$ be an HNN-extension with finite associated subgroups such that $G_1=M_1\rtimes \Z,$ where $M_1$ is a finite subgroup, distinct from $H$ and $K.$ 
Then $G_1$ is 2-ended and hence is not an $\OE$-group. 
However, the results obtained in the previous sections for this $G_1$ remain valid.\end{rem}

\section{Profinite HNN-extensions} \label{secao 4}
\label{profinitehnnextensions}

In this section we will solve the isomorphism problem for profinite HNN-extensions.  

Let $G= \hnn(G_1, H, K, f, t)$ be a profinite HNN-extension. We define the set \begin{equation} \label{N mais}  N^{+}:=\{ gt^{- 1} \in N_{G}(H) \hspace{0.1cm} | \hspace{0.1cm}  \overline{\left\langle g, G_1 \right\rangle}=G\}.\end{equation} 
Consider $g_1\in N_{G_1}(H)$ and $g=g_1t$. Then we have $g_1=gt^{-1}\in N^{+}.$ Therefore \begin{equation} \label{NG1H em N+} N_{G_1}(H) \subset N^{+} \subset N_{G}(H) .\end{equation}

 Further denote by $\overline{N}^{+}$ and $\widetilde{N}^{+}$  their natural images in $\aut(H)$ and $\out(H),$ 
respectively. Note that it is $N^+$ that marks the difference between the abstract and the profinite case, as we will see later, it will be crucial for the calculation of the genus.  

The following proposition will be fundamental to find a bound for the number of isomorphism classes of profinite HNN-extensions. 
 

\begin{prop}\label{nonnormalhnnprof} Let $$G= \hnn(G_1, H, K, f_1, t_1)\  {\rm and}\  B=\hnn(G_1, H, K, f_2, t_2)$$ be profinite HNN-extensions with associated finite subgroups in a profinite $\OE$-group $G_1$. Suppose that $K=H$ or $H$ and $K$ are not conjugate in $G_1$. Then $G\cong B$ if and only if 

\begin{enumerate}
\item [(i)] there exist $\varphi\in \aut(G_1)$ with $\varphi(H)=K, \varphi(K)=H^{g_1}$ for some $g_1 \in G_1$ such that $ f_2^{\varphi^{-1}} = \tau_{g_1^{-1}} \cdot  (f_1\tau_{n^{-1}})^{-1},$ with  $n\in N^{+}$;
\item [(ii)] there exist $\varphi\in \aut(G_1)$ with $\varphi(H)=H, \varphi(K)=K^{g_1}$ for some $g_1 \in G_1$ such that $f_2^{\varphi^{-1}} = \tau_{g_1^{-1}} \cdot f_1 \cdot \tau_{n^{-1}} ,$ with $n\in N^{+}$.
\end{enumerate}
\end{prop}

\begin{proof}

$(\Longrightarrow)$ By Proposition \ref{2a} there exists a continuous  isomorphism \break 
$\psi:B\longrightarrow G$ such that $\psi(G_1) = G_1$, $\psi(H)=K$ and $\psi(K)=H^{g_1}$ or  $\psi(G_1) = G_1$, $\psi(H)=H$ and $\psi(K)=K^{g_1}$ for some $g_1\in G_1$.

First suppose that $H=K$ and observe that then $g_1\in N_{G_1}(H)$. Put $g:=\psi(t_2) \in G$. Since $t_2 \in N_{B}(H)$ we conclude that $gg_1^{-1} \in N_{G}(H)$ and $ 
 \overline{\left\langle  G_1, (gg_1^{-1})^{\epsilon} \right\rangle}=G$, where  $\epsilon \in \{\pm 1\}.$  
Put $n:= (gg_1^{-1})^{\epsilon}t_1^{-1} \in N^{+}$ and observe that for all $h \in H$ one has $$\psi(f_2(h)) = \psi(h)^{\psi(t_2)} = \psi(h)^{(nt_1)^{\epsilon}g_1} =\tau_{g_1^{-1}}(f_1\tau_{n^{-1}})^{\epsilon} \psi(h), \mbox{ where } \, \epsilon = \pm 1.$$ 
Take $\varphi:=\psi|_{_{G_1}} \in \aut_{G_1}(H)$. Then $\varphi f_2\varphi^{-1}=  \tau_{g_1^{-1}}(f_1^{-1}\tau_{n^{-1}})^{\epsilon}$ as required. 

\bigskip
If $H$ and $K$ not conjugate in $G_1$ let's divide it into two cases. 

\textbf{Case 1:} Suppose  $\psi(G_1) = G_1$, $\psi(H)=K$ and $\psi(K)=H^{g_1}$ . 
Put $\psi(t_2) = g \in G$. Note that $ H^{g_1}=\psi(K) =\psi(H^{t_2})=\psi(H)^{\psi(t_2)}=K^g=H^{t_1g}$, from where it follows that $n:=g_1 g^{-1} t_{1}^{-1} \in N_G(H)$, furthermore $n\in N^{+}$. 
Then for all $h \in H,$ we have$$\psi(f_2(h))=\psi(h^{t_2})=\psi(h)^{\psi(t_2)}=\psi(h)^{t_{1}^{- 1}n^{-1}g_1}=(\tau_{g_1^{-1}}\tau_{n} f_1^{- 1}  \psi)(h).$$ Therefore taking $\varphi = \psi|_{_{G_1}} \in \aut(G_1)$ we have  $f_2^{\psi^{-1}} = f_2^{\varphi^{-1}} = \tau_{g_1^{-1}} \cdot \tau_{n} \cdot f_1^{-1}= \tau_{g_1^{-1}} \cdot   (f_1\tau_{n^{-1}})^{-1},$ with $g_1 \in G_1$ and $n\in N^{+}$. 

\textbf{Case 2:} Suppose $\psi(G_1) = G_1$, $\psi(H)=H$ and $\psi(K)=K^{g_1}$. 
Take $\psi(t_2)=g\in G$ and $n:=gg_1^{-1}t_1^{-1} \in N^{+}$. 
Then for all $h\in H$, we have $$\psi(f_2(h))=\psi(h^{t_2})=\psi(h)^{\psi(t_2)}=\psi(h)^{ n t_1g_1}=(\tau_{g_1^{-1}} f_1 \tau_{n^{-1}} \psi)(h).$$ 
So  taking $\varphi=\psi|_{_{G_1}}$ the result follows. 

\bigskip
$(\Longleftarrow)$ Define a map $\psi$  from $B$ to $G$ as follows $$\psi: \left\{
\begin{array}{ccl}
g_1 & \longmapsto  & \varphi(g_1), \forall  g_1 \in  G_1,\\
t_2 & \longmapsto  & (nt_1)^{\epsilon}g_1, \epsilon=\pm 1. \\
\end{array}
\right.$$
If (i) holds, consider $\epsilon=-1$ and if  (ii) holds, consider $\epsilon=1$;  in both cases $\psi$ is an isomorphism.

\end{proof}

\subsection{Isomorphism problem for profinite HNN-extensions} \label{iso profinite HNN}

In this section we will give a bound for the number of isomorphism classes of profinite HNN-extensions with base group $G_1$ being a profinite $\OE$-group and fixed finite associated  subgroups. In the  case when the HNN-extensions have the conjugate associated subgroups  in the base group, some corollaries will also be shown.

Let $G=\hnn(G_1, H, K, f, t)$ be a profinite HNN-extension, where $H, K$ are finite associated subgroups of a profinite OE-group $G_1$. 

  Consider the closed subgroup 
 \begin{equation} \label{gamma H2} 
 \Gamma_{\widehat H}:=G_1 \rtimes \aut_{G_1}(H)\end{equation} of the holomorph $G_1\rtimes \aut(G_1)$ and as in the abstract case sometimes we shall write $\tau_{g_1}$ for the inner automorphism corresponding to an element $g_1$ of $G_1$. 

Denote by $P=\mbox{Subgr}(G_1)$ the set of all closed subgroups of $G_1.$ Note that $\mbox{Subgr}(G_1)$ is a profinite space because if $G_1=\varprojlim_{U\triangleleft_o G}  G_1/U$ then $\mbox{Subgr}(G_1) = \varprojlim_{U\triangleleft_o G}  \mbox{Subgr}(G_1/U)$. 
There is a continuous action of $$\Gamma_{\widehat{H}}=G_1\rtimes \aut_{G_1}(H)$$ on the set $P$ namely \begin{equation} \label{acaoemXp} (g_1,\alpha)\cdot X= \tau_{g_1}(\alpha(X))=\left[\alpha(X)\right]^{g_1^{-1}} \in P, \mbox{ where } (g_1,\alpha) \in \Gamma_{\widehat{H}}, X \in P. \end{equation}
 Let $\Gamma_{\widehat{HK}}$ be the stabilizer of $K$ in $\Gamma_{\widehat{H}}$ for this action on $P$ (see (\ref{acaoemXp})). Then $\Gamma_{\widehat{HK}}$ is a profinite group (see 5.6.3 in \cite{RZ}).

 Now let's fix a finite subgroup $H\leq G_1$ and consider 
$\widehat{\Upsilon} :=\displaystyle\bigcup_{X \leq G_1, |X|<\infty} \Iso(H,X)$.  Similarly to the abstract case, $\Gamma_{\widehat{H}}$ acts on $\widehat{\Upsilon}$.


\begin{lem}\label{elementaryp}
The semidirect product  $\Gamma_{\widehat{H}}=G_1 \rtimes \aut_{G_1}(H)$
acts on the left upon $\widehat{\Upsilon}$ by
\begin{equation} \label{acaosemidiretap} (g_1,\alpha) \cdot f=\tau_{g_1} \alpha f \alpha^{-1} = \tau_{g_1} f^{\alpha^{-1}}, \mbox{ where } (g_1,\alpha) \in \Gamma_{\widehat{H}} \mbox{ 
 and } f \in \widehat\Upsilon.
 \end{equation}
\end{lem}
\begin{proof} It is word by word the same as of Lemma \ref{elementary}. 
\end{proof}

  For $f\in \Iso(H, K)$, we have $(g_1,\alpha) \cdot f \in \Iso\left(H, [\alpha(K)]^{g_1^{-1}}\right) \subset \widehat \Upsilon$ and
the group $\Gamma_{\widehat{HK}}$ leaves $\Iso(H, K)$ invariant with respect  to the action of $\Gamma_{\widehat{H}}$ on $\widehat\Upsilon$ (see (\ref{acaosemidiretap})).  

From now on, let $$G=\hnn(G_1, H, K, f_1, t_1) \mbox{ and } B=\hnn(G_1, H, K, f_2, t_2),$$ be profinite HHN-extensions, where $G_1$ is a profinite OE-group and \break $f_i \in \Iso(H, K)\,(i=1,2)$. Our goal is to describe  when $G$ and $B$ are isomorphic.

Suppose there exist $\psi \in \aut(G_1)$ such that $\psi(K)=H^{g_1}$ and $\psi(H)=K$, for some $g_1\in G_1$.  Define $\iota: \Iso(H,K) \longrightarrow \Iso(H,K)$ by \begin{equation} \label{def iota prof} \iota(f)=\tau_{\psi^{-1}(g_1^{-1})}(f^{-1})^{\psi}=\tau_{\psi^{-1}(g_1^{-1})} \left(f^{\psi}\right)^{-1}, f \in \Iso(H,K).\end{equation}

If $H=K$ we can choose $\psi=id$ and $g_1=1,$ so $\iota$ sends $f$ to $f^{-1}.$ In this situation we have an isomorphism between $\hnn(G_1,H,f)$ and $\hnn(G_1,H,f^{-1})$ (see Lemma \ref{conjugados da normal}). Note that $\iota$ will be a bijection from $\Iso(H, H)=\aut(H)$ to $\aut(H)$. More generally, $\iota$ belongs to the group of symmetries $Sym(\Iso(H, K))$. Since $\Gamma_{\widehat{HK}}$ acts on $\Iso(H, K),$ we  denote by  $\widetilde\Gamma_{\widehat{HK}}$ the natural image of $\Gamma_{\widehat{HK}}$ in $Sym(\Iso(H, K))$ induced by the natural homomorphism $\Gamma_{\widehat{HK}} \longrightarrow Sym(\Iso(H, K)).$ The proof of the following lemma is similar to the proof of Lemma 
\ref{semidiretocomC2}. 

\begin{lem} \label{semidiretocomC2p}  Suppose there exist $\psi \in \aut(G_1)$ that swaps the conjugacy classes of $H$ and $K$.  Then $\iota $ normalizes $\widetilde\Gamma_{\widehat{HK}}$ and so $\langle \widetilde\Gamma_{\widehat{HK}}, \iota  \rangle=\widetilde\Gamma_{\widehat{HK}} \cdot  \left\langle  \iota \right\rangle$ is a subgroup of $Sym(\Iso(H, K))$.   

\end{lem}

The next Lemma is proved similarly as Lemma \ref{carmo}

\begin{lem}\label{carmoprof} The subgroup $\widetilde\Gamma_{\widehat{HK}} \cdot  \left\langle  \iota \right\rangle$ of $Sym(\Iso(H, K))$ does not depend upon the choice of $\psi$.
\end{lem}


\begin{deff}
\label{novaacaopp} Given finite subgroups $H$ and $K$ of $G_1$ we set $$\overline{\Gamma}_{\widehat{HK}}:=\widetilde{\Gamma}_{\widehat{HK}}\cdot \left\langle \iota\right\rangle$$ if there exist $\psi \in \aut(G_1)$ such that $\psi(K)=H^{g_1}$ and $\psi(H)=K$, for some $g_1\in G_1$, and put $$\overline{\Gamma}_{\widehat{HK}}:=\widetilde{\Gamma}_{\widehat{HK}},$$ otherwise. Further let $$\overline{\Gamma}_{\widehat{HK}} \backslash \Iso(H, K)$$ denote the set of orbits of $\Iso(H, K)$ modulo the action  of $\overline{\Gamma}_{\widehat{HK}}$.
\end{deff}

\begin{rem} \label{novaacao2pp} When $H=K$ then $\Iso(H,H)=\aut(H)$. In this case one can define  $\Gamma_{\widehat{H}}=N_{G_1}(H) \rtimes\aut_{G_1}(H)$, $\Gamma_{\widehat{HH}}=\Gamma_{\widehat{H}}$ as being the stabilizer of $H$ of the action of $\Gamma_{\widehat{H}}$ on  $P$ (see (\ref{acaoemXp})) and 
$\overline{\Gamma}_{\widehat{HH}}=\widetilde\Gamma_{\widehat{H}}\cdot C_2$, where $C_2= \left\langle \iota\right\rangle$ acts by inversion on $\aut(H)$. Note that if $\iota \in \widetilde\Gamma_{\widehat{H}}$ then $\overline{\Gamma}_{\widehat{HH}}=\widetilde{\Gamma}_{\widehat{H}},$ otherwise $\overline{\Gamma}_{\widehat{HH}}=\widetilde{\Gamma}_{\widehat{H}} \rtimes C_2.$ 
In particular, $ \overline{\Gamma}_{\widehat{HH}} \backslash \aut(H)$ will denote the set of $ \overline{\Gamma}_{\widehat{HH}}$-orbits of $\aut(H)$.
\end{rem}

The result below is a consequence of Propositions   \ref{nonnormalhnnprof} and the discussion of this subsection. In the profinite case we can only give an estimate for the number of isomorphism classes of profinite HNN-extensions with fixed associated subgroups, as follows.  \begin{thm}\label{profinite isomorphism number} Let $H, K$ be fixed finite subgroups of a profinite $\OE$-group $G_1$. Then the number of isomorphism classes of profinite HNN-extensions \\$\hnn(G_1, H, K, f, t)$ is less than or equal to $$|\overline{\Gamma}_{\widehat{HK}} \backslash \Iso(H, K)|.$$\end{thm}

\subsection{Bounds for the number of isomorphism classes of profinite HNN-extensions with not fixed associated subgroups } \label{quotasprofnonconjugates}

Let $H$ be a finite subgroup of a profinite OE-group $G_1.$ Our goal is to find a bound for the number of isomorphism classes of profinite HNN-extensions of  $G_1$ with associated subgroups isomorphic to $H$. This will be important for the calculation of the genus since by Proposition \ref{profinito normal iso} if $H\not\cong H_1$ then $\hnn(G_1, H, K, f, t)\not\cong \hnn(G_1, H_1, K_1, f_1, t_1)$.

Our considerations in this subsection  will be give analogous to what was done in the abstract case (see Subsection \ref{cota bonita non conjugados}). 
Let $P_H \subset P$ the set of all finite subgroups of $G_1$ isomorphic to $H$. We have an action of $G_1$ on $P_H$  by conjugation. So, given $L\in P_H$ the orbit of $L$ with respect to this action (i.e., the conjugacy class of $L$ in $G_1$) is $$\langle\langle L\rangle\rangle^{G_1}= \{L^{g_1}| g_1 \in G_1\}.$$ Put $$\overline{P}_H= G_1\backslash P_H:=\{ \langle\langle L\rangle\rangle^{G_1} \,|\, L \in P_H\}.$$   

Let $[\overline{P}_H]^2$ be the collection of all unordered pairs of the form $\{X, Y\}$ such that $X, Y \in \overline{P}_H.$ Define similarly to what was done in the abstract case (see Subsection \ref{cota bonita non conjugados}) an action of $\aut(G_1)$ on $[\overline{P}_H]^2$ as follows  

\begin{equation*} \rho: \aut(G_1) \times [\overline{P}_H]^2 \longrightarrow [\overline{P}_H]^2, \end{equation*}  
\begin{equation} \label{acao p}
( \varphi, \{ \langle\langle H_1\rangle\rangle^{G_1} , \langle\langle K_1\rangle\rangle^{G_1} \} ) \mapsto \{ \langle\langle \varphi(H_1)\rangle\rangle^{G_1} , \langle\langle  \varphi(K_1)\rangle\rangle^{G_1}  \}.
\end{equation}

Consider  finite subgroups $H_1$ and $K_1$ of $G_1$ isomorphic to $H$ and $K,$ respectively.  Note that $\{ \langle\langle H\rangle\rangle^{G_1}
, \langle\langle K\rangle\rangle^{G_1}
 \}$ and $\{ \langle\langle H_1\rangle\rangle^{G_1}
, \langle\langle K_1\rangle\rangle^{G_1}
 \}$ belong to the same $\aut(G_1)$-orbit according to action (\ref{acao p}) if and only if  there exist $\psi\in \aut(G_1)$ such that $\psi(H)=H_1$ and $\psi(K)=K_1^{g_1}$ for some $g_1\in G_1.$ 


 By Proposition  \ref{2a}, for each $f\in \Iso(H,K)$ we have $$\hnn(G_1, H, K, f,t) \cong \hnn(G_1, H_1, K_1, \tau_{g_1}f^{\psi^{-1}},t_1).$$

When $H$ and $K$ are fixed finite subgroups of a profinite OE-group $G_1,$ we saw from Theorem \ref{profinite isomorphism number} that the number $n_{_{\widehat{HK}}}$ of isomorphism classes of profinite HNN-extensions of the form $\hnn(G_1, H, K, f, t),$ with $f \in \Iso(H, K)$ satisfies the following $$n_{_{\widehat{HK}}} \leq | \overline{\Gamma}_{\widehat{HK}} \backslash \Iso(H, K)|.$$ 

Denote the set of orbits of action (\ref{acao p}) by $\aut(G_1)\backslash [\overline{P}_H]^2.$  

\begin{thm}\label{profinite problem} Let $H$ and $K$ be finite subgroups of a profinite  $\OE$-group $G_1$. Then the number $\hat \eta$ of isomorphism classes of profinite HNN-extensions $\hnn(G_1, H_1, K_1, f_1)$ such that 
$H_1 \cong H \cong K \cong  K_1$ and $f_1 \in \Iso(H_1, K_1)$ satisfies the following inequalities $$\hat \eta\leq \sum_{\{H, K\}} | \overline{\Gamma}_{\widehat{HK}} \backslash \Iso(H, K)|,$$ where $\{H, K\}$ ranges over  representatives of the $\aut(G_1)$-orbits of $[\overline{P}_H]^2.$
\end{thm}

Note that the set of all finite sugroups of a profinite group is not closed in general in the space of subgroups of it. However, the subset of subgroups of order bounded by a natural number $n$ is closed. This allows to extend the action $\rho$ given by (\ref{acao p}) of $\aut(G_1)$ on $[\overline{P}_H]^2,$ to the set $[\overline{P_n}]^2$ of all unordered pairs of conjugacy classes of  subgroups of $G_1$ of order $\leq n$. Denote by $\aut(G_1) \backslash [\overline{P_n}]^2$ the space of orbits of this action. Then we can extend the preceding theorem by establishing a bound on the number of isomorphism classes of all profinite HNN-extensions of $G_1$ with finite associated subgroups.

\begin{thm}\label{profinite HNN-extensions with finite associated} Let $G_1$ be a profinite  $\OE$-group. Then the number  of isomorphism classes of profinite HNN-extensions $\hnn(G_1, H, K, f, t)$ with finite associated subgroups does not exceed  $$ \sum_{\{H, K\}} | \overline{\Gamma}_{\widehat{HK}} \backslash \Iso(H, K)|,$$ where $\{H, K\}$ ranges over representatives of the orbits $\aut(G_1)\backslash [\overline{P_n}]^2.$ 
\end{thm} 

\begin{rem} We can extend the action of $Aut(G_1)$ to the the set of all finite sugroups of $G_1$ and get the statement as written in Theorem \ref{int general finite associated}, but it differs from the theorem above only when the number of isomorphisms is infinite.\end{rem}

\subsubsection{ Fixed conjugate associated subgroups } \label{quotasprofconjugates}

The objective of this subsection is to find  bounds for  the number of isomorphism classes of profinite HNN-extensions of the form $G=\hnn(G_1, H, K, f)$ with $H,K$ conjugate in $G_1$. For this, we observe that $G$ is isomorphic to a normal profinite HNN-extension, namely, $G\cong \hnn(G_1, H,   \tau_{g_1}f)$ (see Lemma \ref{conjugados da normal}). 

As discussed in the abstract case (see Section \ref{cota bonita conjugados}), we have an action on the right of $\overline{\aut}_{G_1}(H) \times C_2$ on the space of orbits $\overline{N}_{G_1}(H) \backslash \aut(H)$, given by  \begin{equation} \label{acao bunitaap} \left( \overline{N}_{G_1}(H) \cdot f \right) \cdot (\psi, \iota) = \overline{N}_{G_1}(H) \cdot  (f^{-1})^{\psi}=\overline{N}_{G_1}(H) \cdot \left(f^{\psi}\right)^{-1},\end{equation}
where $( \psi, \iota) \in  \overline{\aut}_{G_1}(H) \times C_2$ and $f \in \aut(H).$ We will denote the space of orbits of this action by $\overline{N}_{G_1}(H) \backslash \aut(H)/ ( \overline{\aut}_{G_1}(H) \times C_2).$


Note that $\inn(H)\leq \overline{N}_{G}(H)$ so we can pass to $\out(H)$ using tilde instead of overline. Then we can deduce the bound  for the number  of isomorphism classes of the profinite HNN-extensions $\hnn(G_1, H, K, f, t)$, where $H$ and $K$ are fixed conjugate finite subgroups. 

\begin{thm}\label{profinite isomorphisms} Let $H$ and $K$ be fixed finite conjugate subgroups of a profinite $\OE$-group $G_1$. Then the number $\hat n_H$ of the isomorphism classes of profinite HNN-extensions $G=\hnn(G_1, H, K, f, t)$ is less or equal than $$\left| \widetilde N_{G_1}(H) \backslash \out(H)/ \left( \widetilde{\aut}_{G_1}(H) \times C_2\right) \right|.$$
\end{thm}
\begin{proof}
Follows from Proposition \ref{nonnormalhnnprof}, Theorem \ref{profinite isomorphism number} and the fact that $N_{G_1}(H)\\ \subset N^{+} \subset N_{G}(H)$ (see (\ref{NG1H em N+})).  
\end{proof}

\begin{exa} Let $G= G_1\amalg_{H, \mu} (H\rtimes_{\alpha} \widehat\Z)$ be a  free profinite product with finite amalgamation, where $G_1$ is a profinite $\OE$-group and $\mu$ is the natural monomorphism from $H$ to $H\rtimes_{\alpha} \widehat\Z= H\rtimes_{\alpha} \overline{\left\langle t \right\rangle}.$ This case was missing in the study of free profinite products with finite amalgamation in \cite{BPZ2} .Note that $G\cong \hnn(G_1, H, \gamma, t),$ where $\gamma = \alpha(t)\circ \mu \in \aut(H).$ Then the number of isomorphism classes of such profinite free products with finite amalgamation $G_1\amalg_{H, \mu} (H\rtimes_{\alpha} \widehat\Z)$  with $H$ fixed is less or equal than  
$$n_{_{H}}:=| \widetilde{N}_{G_1}(H) \backslash \out(H)/ ( \widetilde{\aut}_{G_1}(H) \times C_2) |.$$
\end{exa}

\begin{rem} \label{finite-by-cyclic profinite} 
  Let $G=\hnn(G_1, H, K, f, t)$ 
be a profinite HNN-extension with finite associated subgroups such that $G_1=M_1\rtimes \widehat\Z$, where $M_1$ is a finite subgroup, distinct from $H$ and $K.$ Then $G_1$ is not a profinite $\OE$-group. However, the results obtained in this section for this $G_1$  remain valid.


\end{rem} 



\subsection{Profinite vs Abstract}

In this subsection we consider an abstract HNN-extension $\hnn(G_1,H,K,f,t)$ of an $\Oe$-group and its profinite completion $\hnn(\widehat G_1,H,K,f,t)$.

\begin{rem} \label{rem dos normalizadores e densidade}
Since $\inn(H)\leq \overline{N}_{G_1}(H)$ ( respectively $\inn(H) \leq \overline{\aut}_{G_1}(H)$), the statement that $\widetilde{N}_{\widehat G_1}(H)=\widetilde{N}_{ G_1}(H) $ ( respectively $\widetilde{\aut}_{G_1}(H) = \widetilde{\aut}_{\widehat G_1}(H)$) immediately implies $\overline{N}_{\widehat G_1}(H)=\overline{N}_{ G_1}(H)$ ( respectively $\overline{\aut}_{G_1}(H) = \overline{\aut}_{\widehat G_1}(H)$).
\end{rem}

\begin{rem} \label{tilde implica overline} Recall that
if the associated subgroups $H$ and $K$ are conjugate in  $G_1$, say $H^{g_1}=K\,(g_1 \in G_1),$ then $\hnn(G_1,H,K,f,t)$ is isomorphic to a normal  HNN-extension $G=\hnn(G_1, H, f_1) \, \mbox{ with } f_1 \in \aut(H)$  (see Lemma \ref{conjugados da normal}), i.e.,  one can assume $G= \hnn(G_1, H, f),$ where $f \in \aut(H)$ is normal.  If $\overline{N}_{\widehat G_1}(H)=\overline{N}_{G_1}(H)$ (this is the case if for example $N_{G_1}(H)$ is dense in $N_{\widehat G_1}(H)$), then by Proposition \ref{cp1} 

$$\overline{N}_{\widehat G}(H) = \langle \overline{N}_{\widehat G_1}(H),  f\rangle= \langle \overline{N}_{G_1}(H), f\rangle=\overline{N}_{G}(H),$$ $$\widetilde{N}_{\widehat G}(H) = \langle \widetilde{N}_{\widehat G_1}(H), \tilde f\rangle= \langle \widetilde{N}_{G_1}(H), \tilde f\rangle=\widetilde{N}_{G}(H),$$     
where $\tilde f$ is the image of $f$ in $\out(H)$.
\end{rem}

The situations described in the remark above occur in particular when $G_1$ is finite.

Thus having abstract and profinite HNN-extensions $\hnn(G_1,H,K,f,t)$  and  $\hnn(\widehat G_1,H,K,f,t)$ of an $\Oe$-group $G_1$ we have two groups $\overline {\Gamma}_{HK}$ and $\overline{\Gamma}_{\widehat{HK}}$ acting on $\Iso(H,K)$ (the first in fact is a subgroup of the second) and the cardinality of genus is the difference between the the orbits of these action: more precisely, it is the number of $\overline {\Gamma}_{HK}$-orbits lying in $\overline{\Gamma}_{\widehat{HK}}$-orbit of $f$.  Of course the $\overline {\Gamma}_{HK}$-orbits will coincide with $\overline{\Gamma}_{\widehat{HK}}$-orbits if the groups $\overline{\Gamma}_{HK}$ and $\overline{\Gamma}_{\widehat{HK}}$ are equal.

Recall that a pair $(H,K)$ of finite subgroups of $G_1$ (resp. $\widehat G_1$) is a swapping pair if there exists an  automorphism $\psi\in \aut(G_1)$ (resp. $\in \aut(\widehat G_1)$) that swaps the conjugacy classes of $H$ and $K$ in $G_1$ (resp. in $\widehat G_1)$. Of course, to have  $\overline{\Gamma}_{HK}=\overline{\Gamma}_{\widehat{HK}}$ it is necessary for $(H,K)$ to be either a swapping pair or not a swapping pair in both $G_1$ and $\widehat G_1$.  Having this, it suffices to have the density of $\Gamma_{HK}$ in $\Gamma_{\widehat{HK}}$, since $\Iso(H,K)$ is finite. However, this would imply the density of $\aut_{G_1}(H)$  in $\aut_{\widehat G_1}(H)$ and for infinite $G_1$ it happens not in too many cases despite the fact that it is quite difficult to check. The next lemma gives more transparent sufficient conditions that can be quite often checked when $\out(H)$ is small; it  will be used in  Subsection \ref{section 5 1}. 


\begin{lem} \label{orbits non conjugated}
With the assumptions in this subsection, suppose the following conditions are satisfied 

\begin{enumerate}
    \item [i)]  $\widetilde{\aut}_{\widehat{G}_1}(H)=\widetilde{\aut}_{G_1}(H)$ and $[G_1:N_{G_1}(K)]=[\widehat G_1:N_{\widehat G_1}(K)]<\infty$; 
    
    \item [ii)] $(H,K)$ is either a swapping pair in  $G_1$ or not a swapping pair in $\widehat G_1$. 
\end{enumerate}
Then $\overline{\Gamma}_{HK}=\overline{\Gamma}_{\widehat{HK}}$. 
\end{lem}
\begin{proof}
First assume that there exists $\psi \in \aut(G_1)$ such that $\psi(K)=H^{g_1}$ and $\psi(H)=K$, for some $g_1\in G_1$ and observe that $\psi$ extends to an isomorphism $\hat\psi\in \aut(\widehat G_1)$.  Let $f \in \Iso(H, K)$. It follows from the definitions of the respective $\iota$ in (\ref{def iota abst}) and (\ref{def iota prof})  that 
$$\overline{\Gamma}_{\widehat{HK}}\cdot f\ni \iota\cdot f =\tau_{\hat\psi^{-1}(g_1^{-1})} \left(f^{\hat\psi}\right)^{-1}=\tau_{\psi^{-1}(g_1^{-1})} \left(f^{\psi}\right)^{-1}\in \overline{\Gamma}_{HK}\cdot f$$

Take $(\hat{g}, \hat{\alpha}) \in \Gamma_{\widehat{HK}}$ such that $\tilde{w}$ is its image in $\widetilde{\Gamma}_{\widehat{HK}}.$ 
According to item i) 
there exists $\tau_g \in \inn(G_1)$ and $\alpha \in \aut(G_1)$ such that $\tau_{\hat{g}}(k)=\tau_{g}(k), \, \forall k \in K$ and $\hat \alpha|_{_{H}} = \alpha|_{_{H}}$. Let $w$ be the image of $(g, \alpha)$ in $\widetilde{\Gamma}_{HK}.$ In this way, we have  $$\overline{\Gamma}_{\widehat{HK}}\cdot f\ni \tilde{w} \cdot f = \tau_{\hat g} f^{\hat{\alpha}^{-1}} = \tau_{ g} f^{\alpha^{-1}} =  w \cdot f\in \overline{\Gamma}_{HK}\cdot f,\, \forall f \in \Iso(H, K).$$

If the second possibility of item ii) occurs, the proof is similar but simpler since we do not have the action of $\iota$ since $\overline{\Gamma}_{HK}=\widetilde{\Gamma}_{HK}$ and $\overline{\Gamma}_{\widehat{HK}}=\widetilde{\Gamma}_{\widehat{HK}}.$
    
\end{proof}

\begin{rem} \label{altenative conditions}
Note that the condition $[G_1:N_{G_1}(K)]=[\widehat G_1:N_{\widehat G_1}(K)]<\infty$ in item i) of the Lemma \ref{orbits non conjugated} can be replaced by the condition $$N_{G_1}(K)C_{\widehat G_1}(K)=N_{\widehat G_1}(K)$$ together with the hypothesis that $K$ is a conjugacy distinguished subgroup in $G_1$ (that is, if $K$ and $L \leq G_1$ are conjugate in $\widehat G_1$, then they are conjugate in $G_1$), since in both cases it follows that given $\tau_{\widehat g} \in \inn(\widehat G_1)$ there exists $\tau_{g} \in \inn(G_1)$ such that $\tau_{\hat{g}}(k)=\tau_{g}(k), \, \forall k \in K.$
\end{rem}

\section{Genus of HNN-extensions}\label{secao 5}

Let $H$ and $K$ be finite subgroups of an $\Oe$-group $G_1$ and $f\in \Iso(H,K)$. 
Note here that the abstract HNN-extension $\hnn(G_1, H, K, f,t)$  is residually finite and its profinite completion is the proper profinite HNN-extension $\hnn(\widehat{G}_1, H, K, f,t)$ (see \cite{RZ}, Proposition 9.4.3)   and so we can apply the results of the previous sections.

\subsection{Fixed associated subgroups} \label{section 5 1}

Let  $$\mathcal{F}=\{ \hnn(G_1, H, K, f, t) | f \in \Iso(H,K) \mbox{ and } H, K \mbox{ fixed  }\}$$ be the class of all  
abstract HNN-extensions $\hnn(G_1, H,K, f, t)$, where $H$ and $K$ are fixed finite subgroups of $G_1$ and $f \in \Iso(H,K)$. 
Recall that
\begin{equation*}
N^+=\left\{ gt^{- 1} \in N_{\widehat G}(H) \hspace{0.1cm} \left| \hspace{0.1cm}  \overline{\left\langle g, \widehat G_1 \right\rangle}=\widehat G \right. \right\},   
\end{equation*} and  \begin{equation} \label{cadeia de N mais etc} N_{ G_1}(H) \subset N_{\widehat G_1}(H) \subset N^{+} \subset N_{\widehat G}(H). \end{equation} 


Given $G\in \mathcal{F}$, our goal now is to find the cardinality of the set $g(G, \mathcal{F})$.   Define \begin{equation} \label{Omega como uniao} \Omega:=\bigcup_{n\in N^{+}}\overline{\Gamma}_{\widehat{HK}} \cdot (f\tau_{n^{-1}}) \subset \Iso(H, K)\end{equation} as the union of orbits of  $f\tau_{n^{-1}}\,(n \in N^{+}) \in \Iso(H, K)$ with respect to action given in Definition \ref{novaacaopp}. Consider $B=\hnn(G_1,H,K,f_1,t_1)\in \mathcal{F}$. By Proposition \ref{nonnormalhnnprof} $\widehat G \cong \widehat B$ if and only if $f_1$ belongs to $\overline{\Gamma}_{\widehat{HK}}$-orbit of $f\tau_{n^{-1}}$ for some $n\in N^{+}$ if and only if $f_1 \in \Omega.$ 

Thus, we have the following 

\begin{thm}\label{fixed H} Let $H$ and $K$ be finite subgroups of an $\Oe$-group $G_1.$ Then the cardinality of  genus $g(G, \mathcal{F})$ of HNN-extension $G=\hnn(G_1,H,K,f,t)$ is 
 $$|g(G, \mathcal{F})| = \left|\overline{\Gamma}_{HK}  \backslash \Omega\right|\,\,(\mbox{ see formula  } (\ref{Omega como uniao})).$$ 
\end{thm}


Next  we find the cardinality of the genus of $G$ when the subgroups  $\overline{\Gamma}_{\widehat{HK}}$ and $\overline{\Gamma}_{HK}$ of $\Iso(H,K)$ coincide (see Lemma \ref{orbits non conjugated}).   Note that this is always the case if $G_1$ is finite ($G_1 = \widehat{G}_1$) (see Lemma \ref{orbits non conjugated}). 

Let $\kappa: \Iso(H,K)\longrightarrow  \overline\Gamma_{HK}\backslash \Iso(H,K)$ be the natural map and \begin{equation} \label{O conjunto S} R:=\{f\tau_{n^{-1}}\,|\, n \in N^{+}\}.
\end{equation} 

The following theorem gives the precise value for the genus of HNN-extensions; in particular it strengthens  Theorem 5.1 in \cite{BZ}.

\begin{thm} \label{nova conta non conjugadosss} Let $H$ and $K$ be finite subgroups in the $\Oe$-group $G_1$ and $G=\hnn(G_1,H,K,f,t)$ be an HNN-extension. If  $\overline\Gamma_{HK}=\overline\Gamma_{\widehat{HK}}$, then $|g(G, \mathcal{F})| = |\kappa(R)|.$ This is the case if \begin{enumerate}
    \item [i)]  $\widetilde{\aut}_{\widehat{G}_1}(H)=\widetilde{\aut}_{G_1}(H),$  $[G_1:N_{G_1}(K)]=[\widehat G_1:N_{\widehat G_1}(K)]<\infty$ and; 
    
    \item [ii)] $(H,K)$ is either a swapping pair in  $G_1$ or not a swapping pair in $\widehat G_1$. 
\end{enumerate} 
\end{thm}

\begin{proof} Follows directly from Lemma \ref{orbits non conjugated}. \end{proof}

Since $\widetilde{N}^{+} \subset \widetilde{N}_{\widehat G}(H),$ 
the corollary below follows directly from the previous result together with Theorem \ref{classnonnormalabs} and the definition of $R$.  
\begin{cor} \label{normal iguais genero 1 non conjugated} Suppose, in addition to the
hypotheses of Theorem \ref{nova conta non conjugadosss} that $\widetilde{N}_{\widehat G}(H)=\widetilde{N}_{G_1}(H),$
then $|g(G, \mathcal{F})|=1$. In particular, this holds if $\widetilde{N}_{G_1}(H)=\out(H)$.
\end{cor}

\begin{proof} It follows from Lemma \ref{orbits non conjugated} that the actions of $\overline{\Gamma}_{HK}$ and  $\overline{\Gamma}_{\widehat{HK}}$ on $\Iso(H, K)$ correspond to the same set of orbits. The result then follows from Theorem \ref{fixed H}.
\end{proof}


\subsubsection{Conjugate subgroups $H$ and $K$}

If the associated subgroups $H$ and $K$ are conjugate in the base group $G_1$,  one can assume w.l.o.g. that the HNN-extension $G=\hnn(G_1, H, K, f,t)$ is normal (see Lemma \ref{conjugados da normal}). Therefore,  the set $R$ from the preceeding subsection can be seen as a subset of $\aut(H).$  Furthermore,  
if $gt^{-1} \in N^{+}$ then $g \in N_{\widehat G}(H)$. 
Note that the image of $gt^{-1}$ in $\aut(H)$ is given by $\tau_{_{(gt^{-1})^{-1}}}=f^{-1} \tau_{g^{-1 }}.$ Therefore the set $R$ (see formula (\ref{O conjunto S})) can be rewritten as \begin{equation} \label{conjunto S com BZ} R=\{f\tau_{n^{-1}}\,\mid\, n \in N^{+}\}= \{ \tau_{g}^{-1} \mid g \in N_{\widehat G}(H) \hspace{0.1cm} \mbox{and} \hspace{0.1cm}   \overline{ \langle g, \widehat G_1 \rangle}  =\widehat G \}.\end{equation} 




Suppose $\mu\in \aut(H)$ and define $\tilde{\mu}$ to be the image of $\mu$ in $\out(H).$ Define the following subset of $\out(H)$ \begin{equation}\label{cS} \cR= \tilde{f}\widetilde{N}^+ := 
\{\tilde f \tilde \tau_{n}^{-1}\,|\, n \in N^{+}\},\end{equation} where $\cR$ is the natural image of the subset $R$ (see (\ref{O conjunto S})) in $\out(H).$

It follows that $\cR=\tilde{f}\widetilde{N}^+$ is a subset of $\tilde f \widetilde{N}_{\widehat G}(H)$. As discussed in Subsubsection \ref{quotasprofconjugates}, $\widetilde{\aut}_{\widehat G_1}(H)$ acts on $\out(H)$ on the right by conjugation.  
The group $\widetilde{N}_{\widehat G_1}(H)$ acts on $\out(H)$ on the left by multiplication. It follows that $\widetilde{N}_{\widehat G_1}(H) \cdot \tilde{f} \widetilde{N}^+ \cdot \widetilde{\aut}_{\widehat G_1}(H)$ is invariant with respect to the action of $\widetilde{N}_{G_1}(H)$ on the left and with respect to the action of $\widetilde{\aut}_{G_1}(H)$ on the right. In the case of conjugated associated subgroups  in the base group, we have the following formula for the genus.

\begin{thm} \label{1.8 da introducao}  Let $H$ and $K$ be conjugate finite subgroups   in an $\Oe$-group $G_1$ and $G=\hnn(G_1,H,K,f,t)$ be an HNN-extension.  Then $$|g(G, \mathcal{F})| = | \widetilde{N}_{G_1}(H) \backslash \widetilde{N}_{\widehat G_1}(H) \cdot \cR  \cdot \widetilde{\aut}_{\widehat G_1}(H)/( \widetilde{\aut}_{G_1}(H) \times C_2)|,$$ where $C_2$ acts on $\out(H)$ by inversion. 
\end{thm}

\begin{proof} It can be assumed that $G$ is normal and $f \in \aut(H)$ (see Lemma \ref{conjugados da normal}). Let $B=\hnn(G_1,H,f_1,t_1)$ with $f_1 \in \aut(H).$ By Proposition \ref{nonnormalhnnprof}, $\widehat B \cong \widehat G$ if and only if $f_1  \in \overline{N}_{\widehat G_1}(H) \cdot R \cdot \overline{\aut}_{\widehat G_1}(H)$ if and only if  $\widetilde{f_1} \in \widetilde{N}_{\widehat G_1}(H) \cdot \cR \cdot \widetilde{\aut}_{\widehat G_1}(H)$ (since $\inn(H)\leq \overline{N}_{G_1}(H) \cap \overline{\aut}_{G_1}(H)$) and the result follows from Corollary \ref{corhnn}. 
\end{proof}

Now, let 
\begin{equation} \label{kappa function} \tilde\kappa: \out(H)\longrightarrow \widetilde{N}_{G_1}(H) \backslash \out(H)/  ( \widetilde{\aut}_{G_1}(H) \times C_2)\end{equation}
be the natural map (see Corollary \ref{corhnn}).
Then we have the following 


\begin{cor} \label{nova conta conjugadoss} Let $H$ and $K$ be conjugate finite subgroups of an $\Oe$-group $G_1$ and $G=\hnn(G_1,H,K,f,t)$ be an HNN-extension. Suppose that $\widetilde{\aut}_{\widehat G_1}(H)\\=\widetilde{\aut}_{G_1}(H)$ and $\widetilde{N}_{\widehat G_1}(H)=\widetilde{N}_{ G_1}(H).$  
Then $$|g(G, \mathcal{F})| = |\tilde \kappa(\cR)|=| \widetilde{N}_{G_1}(H) \backslash \cR/  ( \widetilde{\aut}_{G_1}(H) \times C_2)|.$$ 
\end{cor}

\begin{proof} The result follows from Theorem \ref{1.8 da introducao}. 
\end{proof}

\begin{cor} \label{cotas n} Let $H$ be a finite subgroup of an $\Oe$-group $G_1$ and $G=\hnn(G_1,H,f,t)$  be a normal HNN-extension.  Suppose $\widetilde{N}_{\widehat G_1}(H)=\{1\}$ (for example if $H$ is malnormal in $\widehat G_1$) and $\widetilde{\aut}_{ G_1}(H)=\widetilde{\aut}_{\widehat{G}_1}(H)$ (this is the case if $|G_1|<\infty$). Then 

$$|g(G, \mathcal{F})|\leq \left\{\begin{array}{cc}
  1,&\mbox{ if } |\tilde f| \leq 2,\\
\displaystyle\frac{\phi(|\tilde f|)}{2} , &\mbox{ if } |\tilde f| \geq 3,
\end{array}\right.$$
where $\phi$ is Euler's function.  
\end{cor}

\begin{proof}
By Proposition \ref{cp1} $\widetilde{N}_{\widehat{G}}(H)=\left \langle \widetilde{N}_{\widehat{G }_1}(H), \tilde{f} \right\rangle=\left \langle \tilde{f} \right\rangle.$ Note that  $\mathcal{R}=\left\{ \tilde{\tau}_{g^{-1}}\,|\, \tilde{\tau}_{g^{-1}} \mbox{ is a generator of } \left \langle \tilde{f}\right \rangle \right\}.$ Therefore, the bound can be obtained by considering only the action of $C_2$ on $\mathcal{R}$ as in Corollary \ref{nova conta conjugadoss}. 
\end{proof} 

\begin{rem}\label{H subgroup cyclic}
In the preceding result, if $\aut(H)$ is abelian (for example, when $H$ is cyclic) then $|g(G, \F)|=\frac{\phi(|\tilde f|)}{2}$ if $|\tilde f|\geq 3.$      
\end{rem} 

Denote by $tor(G)$ the subset of elements of finite order in a given group $G.$ 

\begin{exa} Suppose that $G_1$ is a finitely generated abelian group which is not virtually infinite cyclic and let $H$ be one of its finite subgroups. By Proposition \ref{virtually cyclic} $G_1$ is an $\Oe$-group. Let $G=\hnn(G_1,H,f,t)$ be a normal HNN-extension. Since $G_1$ is finitely generated abelian $\widetilde{N}_{G_1}(H)= \widetilde{N}_{\widehat{G }_1}(H)=\{1\}$ and furthermore $\widetilde{\aut}_{ G_1}(H)=\out(H)=\widetilde{\aut}_{\widehat{G}_1}(H).$ It follows from Corollary \ref{cotas n} that the genus satisfies 
$$|g(G, \mathcal{F})|\leq \left\{\begin{array}{cc}
  1,&\mbox{ if } |\tilde f| \leq 2,\\
\displaystyle\frac{\phi(|\tilde f|)}{2}, &\mbox{ if } |\tilde f| \geq 3.
\end{array}\right.$$ 

If $G_1$ is virtually infinite cyclic and $H \neq tor(G_1)$ the same result is valid by Remarks \ref{finite-by-cyclic} and \ref{finite-by-cyclic profinite}.  
For $H=tor(G_1)$ we have that the HNN-extension is of the form $G=H \rtimes F,$ where $F$ is a free group and the genus in this case was calculated in Theorem 3.26 in \cite{GZ}. 
\end{exa}



The second part of the result below, which deals with genus equal to 1, is similar to what was done in Theorem 3.1 in \cite{BZ} when $G$ is restricted to the class $\mathcal{F}$ and the finite associated subgroups are conjugate in $G_1,$ that is, when the HNN-extension can be considered normal.

\begin{cor} \label{autG classe} 
Let $H$ and $K$ be conjugate finite subgroups of an $\Oe$-group $G_1$  and $G=\hnn(G_1, H, K, f, t)$ be an HNN-extension with $f \in \Iso(H, K).$ Suppose $\widetilde{\aut}_{\widehat G_1}(H)=\widetilde{\aut}_{G_1}(H)$ and $\widetilde{N}_{\widehat G_1}(H)=\widetilde{N}_{ G_1}(H)$ (for example $|G_1|<\infty$).  Then $$|g(G, \mathcal{F})| \leq | \tilde\kappa(\widetilde{N}_{G}(H))|.$$ In particular, if $\widetilde{N}_{\widehat G}(H)=\widetilde{N}_{G_1}(H)$ (this is the case if $\tilde f\in \widetilde{N}_{ G_1}(H),$ where $\tilde f$ is the image of $f$ in $\out(H)$) or $|\out(H)|\leq 2,$ then $|g(G, \mathcal{F})|=1$.
\end{cor}
\begin{proof} Note that $\widetilde N_{\widehat G}(H)=\langle \widetilde N_{\widehat G_1}(H), \tilde f\rangle=\widetilde N_G(H)$ in this case. So  the first part follows directly from Corollary \ref{nova conta conjugadoss} and the fact that $N^{+} \subset N_{\widehat G}(H),$ because the genus is less than or equal to $$| \tilde\kappa(\tilde f \cdot \widetilde{N}_{G}(H))|=| \tilde\kappa( \widetilde{N}_{G}(H))|.$$ The second statement follows directly from the first. 
\end{proof}

\begin{rem} \label{case polycyclic}
Let $G_1$ be a polycyclic-by-finite group and $H$ a finite subgroup of $G_1$ such that $\widetilde{\aut}_{\widehat G_1}(H)= \widetilde{\aut}_{G_1}(H)$. Consider  a normal HNN-extension $G=\hnn(G_1, H, f, t)$. Then the cardinality of the genus of $G$ in $\mathcal{F}$  satisfies $$|g(G, \F)| \leq  | \tilde\kappa( \widetilde{N}_{G}(H))|,$$ by Proposition \ref{cp1},  Corollaries  \ref{8} and \ref{autG classe} since $\widetilde N_{\widehat G}(H)=\widetilde N_{G}(H)$. In particular, the result holds if $G_1$ is finite. 
\end{rem}

The  result, which provides more effective conditions for calculating the genus of normal HNN-extensions, will be given in the following corollary. Note that it is a generalization of Corollary 5.5 in \cite{BZ}.  


\begin{cor} \label{genus cotas faceis}
Let $H$ be a finite subgroup of an $\Oe$-group $G_1$ and $G=\hnn(G_1,H,f,t)$ be an HNN-extension, where $f \in \aut(H)$. Suppose that $\widetilde{\aut}_{\widehat{G}_1}(H)=\widetilde{\aut}_{G_1}(H)$, $\widetilde{N}_{\widehat G_1}(H)=\widetilde{N}_{ G_1}(H)$   and moreover,  one of the following conditions is satisfied \begin{enumerate}
    \item [i)] $f \in \overline{\aut}_{G_1}(H);$

    \item [ii)] $\overline{N}_{ G_1}(H)=\inn(H);$

    \item [iii)] $H \leq Z(G_1).$
\end{enumerate}  
Then \begin{enumerate}    \item [a)] $|g(G, \F)|=1,$ if $|\widetilde{N}_{ G}(H)/\widetilde{N}_{ G_1}(H)|\leq 2$;
    \item [b)] 
 $|g(G, \F)|\leq \frac{\phi\left(|\widetilde{N}_{G}(H)/\widetilde{N}_{G_1}(H)|\right)}{2},$ if $\left| \frac{\widetilde{N}_{ G}(H)}{\widetilde{N}_{ G_1}(H)}\right| \geq 3,$ where $\phi$ is the Euler function.
\end{enumerate} 

\end{cor}

\begin{proof} By Remarks \ref{rem dos normalizadores e densidade} and 
\ref{tilde implica overline} we have  $$\overline{N}_{\widehat G}(H) = \langle \overline{N}_{\widehat G_1}(H),  f\rangle= \langle \overline{N}_{G_1}(H),  f\rangle=\overline{N}_{G}(H)$$ and $\widetilde{N}_{\widehat G}(H) = \widetilde{N}_{G}(H)=\langle  \widetilde{N}_{G_1}(H),  \tilde{f}\rangle.$ Since $\overline{N}_{G_1}(H) \triangleleft \overline{\aut}_{G_1}(H),$ then for any of the hypotheses i)-iii) of this corollary we have that $\overline{N}_{G_1}(H) \triangleleft \overline{N}_{\widehat G}(H)=\overline{N}_{G}(H),$  whence it follows that $\widetilde{N}_{G_1}(H) \triangleleft \widetilde{N}_{\widehat G}(H)=\widetilde{N}_{G}(H).$ Therefore, the quotients $\overline{N}_{ G}(H)/\overline{N}_{ G_1}(H) \cong \widetilde{N}_{ G}(H)/\widetilde{N}_{ G_1}(H)$ are finite cyclic. Let $\tilde{\zeta}: \widetilde{N}_{G}(H) \longrightarrow \widetilde{N}_{ G}(H)/\widetilde{N}_{ G_1}(H)$ be the canonical epimorphism. By the definition of $R$ and $\mathcal{R}$ in (\ref{conjunto S com BZ}) and (\ref{cS}), we have that  $\tilde{\zeta}(\mathcal{R})$ is the set of all generators of the cyclic group $\widetilde{N}_{ G}(H)/\widetilde{N}_{ G_1}(H).$ Since $\widetilde{N}_{G_1}(H) \backslash \cR \subset \widetilde{N}_{G_1}(H) \backslash \widetilde{N}_{G}(H)$, we obtain the result by Corollary \ref{nova conta conjugadoss} by looking only at the orbits obtained by the action of $C_2$ on the generator set $\widetilde{N}_{G_1}(H) \backslash \cR.$
\end{proof}

\subsection{Genus 1}

In this subsection we will give some sufficient conditions for  the genus of HNN-extensions with  an $\Oe$ base group and finite associated subgroups to have only one element. Of course, if  $\widetilde{N}_{G_1}(H)=\out(H)$, then $|g(G, \mathcal{F})| =1
$ by Corollary 
\ref{normal iguais genero 1 non conjugated} and so we shall omit this condition from the list  below to avoid repeatition. The following results give a generalization of Theorem 1.1 in \cite{BZ}. 


We start the case of normal HNN-extension (equivalently $H$ and $K$ are conjugate in the base group).

\begin{thm}\label{genus 1 normal}
Let $H$ be a finite subgroup of an $\Oe$-group $G_1$ and $$G=\hnn(G_1, H, f, t)$$ be a normal HNN-extension, i.e. $f \in \aut(H).$ 
Suppose that   $\widetilde{\aut}_{\widehat{G}_1}(H)=\widetilde{\aut}_{G_1}(H)$ and $\widetilde N_{G_1}(H)=\widetilde N_{\widehat G_1}(H)$. If $\widetilde{N}_{\widehat G}(H)=\widetilde{N}_{G_1}(H)$ (this is the case if $f\in \overline N_{G_1}(H)$), then $|g(G,\F)|=1$. 
\end{thm}
\begin{proof} Follows from Theorem \ref{1.8 da introducao}.
\end{proof}



Now we state theorem for not 
normal HNN-extensions. Note that the particular case in item iv) of Theorem \ref{genus 1 intr} automatically enters in the next theorem,  since if $H$ or $K$ is central in $G_1,$ then $H\neq K$ implies non-conjugacy of $H$ and $K$.

\begin{thm}\label{genus 1}
Let $H$ and $K$ be  finite non-conjugate subgroups of an $\Oe$-group $G_1$ and $G=\hnn(G_1, H, K, f, t)$ be an HNN-extension.
Suppose that $\widetilde{\aut}_{\widehat{G}_1}(H)=\widetilde{\aut}_{G_1}(H)$,   $[G_1:N_{G_1}(K)]=[\widehat G_1:N_{\widehat G_1}(K)]<\infty$ and 
     either $(H,K)$ is a swapping pair in $G_1$ or  not a swapping pair in $\widehat G_1$. 
 Then there is only one element in the genus of $G$ in $\F,$ if one of the following conditions is satisfied:
\begin{itemize} 
\item[i)]   
     $f$ extends to an automorphism of $G_1;$    
    \item [ii)] $N_{G_1}(K)=K \cdot C_{G_1}(K)$ or $N_{G_1}(H)=H \cdot C_{G_1}(H)$ (in particular,  $H$ or $K$ is central in $G_1$).  
\end{itemize}
\end{thm}



\begin{proof}
First note that the hypothesis $[G_1:N_{G_1}(K)]=[\widehat G_1:N_{\widehat G_1}(K)]<\infty$ implies that $N_{G_1}(K)$ is dense in $N_{\widehat G_1}(K),$ and as $$N_{G_1}(H)=N_{G_1}(K^{t^{-1}})=\left(N_{G_1}(K)\right)^{t^{-1}}$$ it follows that $N_{G_1}(H)$ is dense in $N_{\widehat G_1}(H)$. Therefore  $\overline N_{G_1}(K) = \overline N_{\widehat G_1}(K)$ and $\overline N_{G_1}(H) = \overline N_{\widehat G_1}(H).$ Since $f(h)=h^t=k$ if and only if $f(h)^{t^{-1}}=h=k^{t^{-1}}, \forall h \in H, k \in K,$ one has $\inn(H)=f^{-1} \cdot \inn(K) \cdot f.$  By Proposition \ref{cp1} \begin{equation} \label{genus 1 prova}  \overline{N}_{\widehat G}(H) = \langle \overline{N}_{\widehat G_1}(H), \overline{N}_{\widehat G_1}(K)^{f} \rangle=\langle \overline{N}_{G_1}(H), \overline{N}_{G_1}(K)^{f} \rangle=\overline{N}_{G}(H)\end{equation} and so   $\widetilde{N}_{\widehat G}(H) = \widetilde{N}_{G}(H)=\langle  \widetilde{N}_{G_1}(H),  \widetilde{N}_{G_1}(K)^{\tilde f}\rangle.$ We need to find the image of $\overline{N}_{\widehat G_1}(H)$ and $\overline{N}_{\widehat G_1}(K)^{f}$ in $\aut(H).$ We will divide the rest of the proof into the following cases.

\textbf{Case 1:} Suppose that $f$ extends to an automorphism of $G_1.$ In this case, if $\hat{g}_1 \in N_{\widehat G_1}(K)$,  then $\tau_{\left( \hat{g}_1^{t^{-1}}\right)^{-1}}= f^{-1} \tau_{g_1^{-1}}f=\tau_{f^{-1}(g_1^{-1})}$ for some $g_1 \in N_{G_1}(K)$, since $\overline N_{G_1}(K) = \overline N_{\widehat G_1}(K)$ and since $f$ extends to $G_1$. It follows that $\tau_{f^{-1}(g_1^{-1})} \in \overline{N}_{G_1}(H).$ Therefore $\overline{N}_{\widehat G_1}(K)^{f} \subset \overline{N}_{G_1}(H),$ whence it follows by (\ref{genus 1 prova}) that $\widetilde{N}_{\widehat G}(H)=\widetilde{N}_{G_1}(H),$ thus the result follows from the Corollary \ref{normal iguais genero 1 non conjugated}. 

\textbf{Case 2:} Suppose $N_{G_1}(K)=K \cdot C_{G_1}(K).$ In this case, we have that $$\overline{N}_{G_1}(K)=\inn(K)=f\cdot \inn(H) \cdot f^{-1}=\overline{N}_{\widehat G_1}(K),$$  thus $\overline{N}_{\widehat G_1}(K)^{f} \subset \inn(H) \subset \overline{N}_{G_1}(H),$ and the result follows from (\ref{genus 1 prova}) and the Corollary \ref{normal iguais genero 1 non conjugated} since $\widetilde{N}_{\widehat G}(H)=\widetilde{N}_{G_1}(H).$ 

\textbf{Case 3:} Suppose $N_{G_1}(H)=H \cdot C_{G_1}(H).$ In this case $$\overline{N}_{G_1}(H)=\inn(H)=f^{-1}\cdot \inn(K) \cdot f\subset f^{-1} \cdot \overline{N}_{G_1}(K) \cdot f.$$ It follows from (\ref{genus 1 prova}) that $\overline{N}_{\widehat G}(H)=\overline{N}_{G_1}(K)^{ f}$ and therefore $\widetilde{N}_{\widehat G}(H)=\widetilde{N}_{G_1}(K)^{\tilde f}.$ Note that $R=f\overline{N}^+\subset \overline{N}_{ G_1}(K)f$ since $\overline{N}^+\subset \overline{N}_{\widehat G}(H)$ (see equations (\ref{cadeia de N mais etc}) and  (\ref{O conjunto S})). The result follows from Theorem \ref{nova conta non conjugadosss} since $|g(G, \F)|=|\kappa(R)| \leq |\kappa(\overline{N}_{ G_1}(K)f)|=1$ because any HNN-extension of the form $B=\hnn(G_1, H,\\ K, \tau_{g_1}f, t')$ with $g_1 \in N_{G_1}(K)$ are isomorphic by Theorem \ref{classnonnormalabs}. 
\end{proof}


The result below is an immediate consequence of item ii) of the previous result because if $G_1$ is finitely generated abelian then $\overline{\aut}_{G_1}(H)=\overline{\aut}_{\widehat{G_1}}(H)$ and the other conditions of the previous theorem are also satisfied. Note that if $G_1$ is virtually infinite cyclic and $H \neq tor(G_1)$ the result still follows from Theorem \ref{genus 1} ii) and Remark  \ref{finite-by-cyclic}. For $H=tor(G_1)$ we have that the HNN-extension is of the form $G=H \rtimes F,$ where $F$ is a free group and the genus in this case was calculated in Theorem 3.26 in \cite{GZ}. This result generalizes Corollary 3.3 in \cite{BZ}.

\begin{cor} \label{fin ger abelian} Suppose that $G_1$ is a finitely generated abelian group and $H\neq K$ its finite subgroups. Then
 $|g(\hnn(G_1, H, K, f, t), \mathcal{F})|=1$.  
\end{cor}

Recall that an indecomposable (into a free product) finitely generated Fuchsian group $\Gamma$ has the presentation of
the form: 
$$
\Gamma=\langle\,a_1,b_1,\ldots,a_n, b_n,\, c_1,\ldots, c_t\, \mid$$
$$ c_1^{e_1}=\ldots=
c_t^{e_t}=1,c_1^{-1}\cdot\ldots \cdot
c_t^{-1}\cdot[a_1,b_1]\cdot\ldots\cdot [a_n,b_n]=1 \, \rangle,$$  where $n,\,
t\geq 0$ and $e_j>1$ for $j=1,\ldots, t$ (see \cite{newman}).

\begin{exa}\label{grupo fuchsian}  Let $\Gamma$ be indecomposable into a free product finitely generated Fuchsian group and $H \neq \{1\}$ is one of its finite subgroup. Let $G_1=K\times \Gamma$ such that $H \cong K$. We first observe that indecomposable into a free product Fuchsian groups are virtually compact surface groups and so they are $\Oe$-groups by Proposition \ref{examples} (i). Therefore by Proposition \ref{extensao de OE group} $G_1$ is an $\Oe$-group. Consider $G=\hnn(G_1,H, K, f, t)$ where $f\in \Iso(H, K).$ Note that $H \neq K$ and therefore $G$ is not normal HNN-extension. By Corollary 7.4 of \cite{BPZ2}  $\overline{\aut}_{\Gamma}(H)= \aut(H)=\overline{\aut}_{\widehat \Gamma}(H)$ therefore we have $$\widetilde{\aut}_{\Gamma}(H)=\widetilde{\aut}_{G_1}(H)= \out(H)=\widetilde{\aut}_{\widehat \Gamma}(H)=\widetilde{\aut}_{\widehat G_1}(H).$$ Now note that $H$ is not a normal subgroup of $\widehat \Gamma$ by Corollary 5.2 in \cite{BCR16}, thus there is no $\hat \psi \in \aut(\widehat G_1)$ such that $\hat \psi(H)=K$ because $K$ is normal in $\widehat G_1.$ Furthermore, since $K \leq Z(\widehat G_1)$, it follows that $N_{G_1}(K)=G_1$ and $N_{\widehat G_1}(K)=\widehat G_1.$ Thus all the hypotheses of Theorem \ref{genus 1} are satisfied, and it follows from item ii) that $|g(G, \mathcal{F})|=1.$ 


\end{exa}


\begin{exa} Let $M$ be an $\Oe$-group and $H$ be a finite group such that $\out(H) $ is trivial (for example this occurs for symmetric groups $ S_n$ with $n \neq 6$ and Mathieu simple groups $M_{11}, M_{23}$ and $M_{24}$). 
Let $G_1=M\times H.$ From Proposition \ref{extensao de OE group} $G_1$ is an $\Oe$-group. Consider $G=\hnn(G_1,H,f,t)$ where $f\in \aut(H)$. 
Note that $\overline{\aut}_{G_1}(H)=\aut(H).$ Indeed, take $\alpha\in \aut(H)$ and observe that $\psi:=(id,\alpha) \in \aut(G_1)$. Thus $\overline{\aut}_{G_1}(H)=\aut(H)=\overline{\aut}_{\widehat G_1}(H)=\inn(H),$ and therefore we have $\widetilde{N}_{\widehat G}(H)=\widetilde{N}_{G_1}(H)=\widetilde{N}_{\widehat G_1}(H)=\widetilde\aut_{G_1}(H)=\widetilde{\aut}_{\widehat G_1}(H)=\out(H)=\{1\}.$ 
Thus all hypotheses of Theorem \ref{genus 1 normal} are satisfied, so $|g(G, \mathcal{F})|=1.$
\end{exa}

\begin{rem}
In the previous example, if $G_1=M\times H,$ where $H=S_6,$ then $\overline{\aut}_{G_1}(H)=\aut(H)=\overline{\aut}_{\widehat G_1}(H)$ consequently $\widetilde{\aut}_{G_1}(H)=\out(H)=\widetilde{\aut}_{\widehat G_1}(H) \cong C_2$ and $\widetilde{N}_{\widehat G_1}(H)=\widetilde{N}_{ G_1}(H)$ (this occurs since $N_{G_1}(H)= G_1$  is dense in $N_{\widehat G_1}(H)=\widehat G_1$), therefore it follows from Corollary \ref{autG classe} that $|g(G, \mathcal{F})|=1.$ Similarly, the same result is concluded for the simple Mathieu groups $M_{12}$ and $M_{22}.$ 
\end{rem}

We finish this subsection with an example of HNN-extension where the base group is virtually infinite 
cyclic which is relevant by Remark \ref{finite-by-cyclic}.

\begin{exa} Let $H$ and $K$ be  isomorphic finite cyclic groups and $G_1= (H\rtimes_{\varphi}\mathbb{Z})\times K$ where  $\varphi:\mathbb{Z}\longrightarrow \aut(H)$ is a homomorphism satisfying   $\varphi(1)\neq id.$ Consider $G=\hnn(G_1,H,K,f,t)$ where $f\in \Iso(H,K)$. Observe that $\widetilde{\aut}_{G_1}(H)=\overline{\aut}_{G_1}(H)=\out(H)=\aut(H)=\overline{\aut}_{\widehat G_1}(H)=\widetilde{\aut}_{\widehat G_1}(H).$ Indeed, let $\beta$ be an automorphism of $H$. The map $\theta:G_1\longrightarrow G_1$ defined by $\theta|_{H}=\beta$, $\theta|_{K}=id$, $\theta|_{\mathbb{Z}}=id$ is an automorphism of $G_1$ that lifts $\beta$. Now there is no $\hat \psi\in \aut(\widehat G_1)$ such that $\hat \psi(H)=K$ since $K$ is central in $G_1$ but $H$ is not. Note that $H \neq K,$ thus it follows from Theorem \ref{genus 1} $ii)$ that $|g(G, \mathcal{F})|=1$.
 
\end{exa}

\section{Non-fixed associated subgroup} \label{non fixed}

Let $\A$ be the class of all finitely generated residually finite accessible groups whose vertex groups of its JSJ-decomposition are $\Oe$-groups (see  on p. 224 of \cite{BPZ} for more details). 

Let $H$ and $K$ be finite subgroups of an $\Oe$-group $G_1$ and $$G=\hnn(G_1, H,\ K, f, t)$$ be an HNN-extension. Let $B\in \A$ such that $\widehat G \cong \widehat B.$ Then by \cite[Theorem 1.1]{BPZ} $B$ is an HNN-extension of the form $\hnn(B_1, H_1, K_1, f_1, t_1)$ with $\widehat B_1\cong \widehat G_1, H \cong H_1\cong K \cong K_1,$ and furthermore $B_1 \in \mathcal{A}$. The class $\F(G,G_1,\A)$ is the subclass of all such $B \in \mathcal{A}$ with $B_1\cong G_1$. Thus we assume from now on that $B_1=G_1$, i.e. consider all HNN-extensions of the form $\hnn(G_1, H_1, K_1, f_1, t_1)$ whose associated  subgroups are finite. 

 We will denote the genus of $G$ in the class $\F(G,G_1,\A)$ as $g(G, G_1, \mathcal{A}).$ Our goal in this section is to find formulas and/or bounds for the cardinality of the genus $g(G, G_1, \mathcal{A})$ of a given HNN-extension with $G_1$ fixed in the class $\A$.

Recall that $S_H$ (respectively $P_H$) is the collection of all finite subgroups of $G_1$ isomorphic to $H$ (respectively subgroups of $\widehat{G}_1$ isomorphic to $H$).  Note that $G_1$ (resp. $\widehat{G}_1$) acts naturally on $S_H$ (resp. $P_H$) by conjugation and as defined in Subsections \ref{cota bonita non conjugados} and \ref{quotasprofnonconjugates}, respectively, denote by 
$$\overline{S}_H:=G_1 \backslash S_H\ \  (\textrm{resp.}\  \overline{P}_H:= \widehat{G}_1 \backslash P_H),$$ the orbit spaces of these actions (the sets of conjugacy classes of subgroups of $S_H$ (and of $P_H$ respectively)). 

Denote by $[S_H]^2$ and $[P_H]^2$ the respective sets of unordered pairs of elements of $S_H$ and $P_H.$
The respective sets of unordered pairs of these orbits were defined as $[\overline{S}_H]^2$ and $[\overline{P}_H]^2.$ Furthermore, let us remember that $\aut(G_1)$ (resp. $\aut(\widehat{G}_1)$) acts naturally on $[\overline{S}_H]^2$ (resp. $[\overline{P}_H]^2$) (see (\ref{acao epsilon}) and (\ref{acao p})) (see Subsections \ref{cota bonita non conjugados} and \ref{quotasprofnonconjugates} for more details). 

 Their respective stabilizers will be denoted as $\aut_{G_1}(H_1, K_1)$ and $\aut_{\widehat G_1}(H_1,\\ K_1),$ respectively. Note that $\aut_{G_1}(H_1, K_1)$ is the subgroup of automorphisms of $G_1$ that leaves the conjugacy classes of $H_1$ and $K_1$ in $G_1$ invariant 
 up to a swap 
 and similarly $\aut_{\widehat G_1}(H_1, K_1)$ is the subgroup of automorphisms of $\widehat G_1$ that leaves the conjugacy classes of $H_1$ and $K_1$ in $\widehat G_1$ invariant 
 up to a swap.

\bigskip
 Let $G=\hnn(G_1, H, K, f, t)$ be an abstract HNN-extension. Denote by $\aut(\widehat G_1)\{H,K\}  \subset [\overline{P}_H]^2$ the $\aut(\widehat G_1)$-orbit of $\{ \langle\langle H\rangle\rangle^{\widehat G_1} , \langle\langle K\rangle\rangle^{\widehat G_1} \}$ (see Section \ref{quotasprofnonconjugates} for the notation). Note that $\aut(G_1)$ (seen as a subgroup of $\aut(\widehat G_1)$) acts on $\aut(\widehat G_1)\{H,K\} $. If $\aut(\widehat G_1)\{H,K\}$ is finite, then   the number of $\aut(G_1)$-orbits of $\aut(\widehat G_1)\{H,K\}$ is
$$|\aut(G_1)\backslash \aut(\widehat G_1)/\aut_{\widehat G_1}(H, K)| < \infty.$$
 
 We denote by $\mathscr{L}_{_{HK}}$ the subset of $\aut(G_1)$-orbits of $\aut(\widehat G_1)\{H,K\}$  which has representatives in $[S_H]^2,$ i.e, satisfies the following property: $$\{ \langle\langle H_1 \rangle\rangle^{\widehat G_1} , \langle\langle K_1\rangle\rangle^{\widehat G_1} \} \in \mathscr{L}_{_{HK}}$$ if and only if  there exists $g_i\in \widehat G_1 (i=1,2)$ such that $\{H_1^{g_1}, K_1^{g_2}\} \in [S_H]^2$. 
 
We denoted by $k(G_1,H,K)$ the cardinality of set  $\mathscr{L}_{_{HK}}$. The estimate for $k(G_1,H,K)$ is given by 
$$k(G_1,H,K)\leq |\aut(G_1)\backslash \aut(\widehat G_1)/\aut_{\widehat G_1}(H, K)|.$$
Furthermore, if $S=P$ (in particular $S_H = P_H$)  then the inequality becomes an equality. This is true for example if $G_1$ is finitely generated nilpotent, because $tor(G_1)$ is finite  and $tor(G_1)=tor(\widehat{G}_1)$ (see Corollary 4.7.9 in \cite{RZ}).

\begin{rem} \label{os ls}
Note that $\aut(G_1)$ acts naturally on $S_H$ by $\psi \cdot L=\psi(L), \, \forall L \in S_H \mbox{ and } \psi \in \aut( G_1).$ Denote by $\aut(G_1) \cdot L$ the orbit of $L$ with respect to this action, similarly write $\aut(\widehat{G}_1) \cdot L$ for the orbit of the respective action when $L$ is in $P_H.$ Let $l_{H}, l_{K}$ be the cardinality of the $\aut(G_1)$-orbits of $S_H \cap (\aut(\widehat{G}_1) \cdot H)$ and $S_H \cap (\aut(\widehat{G}_1) \cdot K)$, respectively. Then $k(G_1, H, K) \leq l_{H} \cdot l_{K}.$ So when $l_H=l_K=1$ we have  $k(G_1, H, K)=1.$       
\end{rem}

\begin{rem} Note that if $G_1$ has only finitely many conjugacy classes of finite subgroups, then  $k(G_1,H,K)$ is finite. Arithmetic groups   \cite{borelarith} and many groups of geometric nature as well as  mapping class groups of any surface of finite type \cite[Theorem 6]{Brid}, the automorphism group and the outer automorphism group of a free group of finite rank (see \cite{culler}, \cite{zimm} and p. 19 in \cite{vog}),  hyperbolic groups   \cite[Theorem 3.2, p. 459]{bridcurvature},  
groups that act properly and cocompactly by isometries on a $CAT(0)$ space \cite[Corollary 2.8 (2), p. 179]{bridcurvature}, share this property.    
\end{rem}

Now  we can state the  theorem which gives the precise cardinality for the genus $g(G,G_1,\A)$, where $G=\hnn(G_1, H, K, f, t)$. If one of the indices in the following summation is infinite or if any of the summands is infinite, then we say that the genus is infinite.

\begin{thm}\label{genus hnn nao fixos} Let $G_1$ be an $\Oe$-group and $G=\hnn(G_1, H, K, f, t)$ be an HNN-extension with finite associated subgroups. Then the cardinality of genus $g(G,G_1,\A)$ of $G$  equals to $$|g(G,G_1,\A)|=\displaystyle\sum_{\{H_1,K_1\}} \left| g\left(\hnn(G_1, H_1, K_1, f_1, t_1), \F\right)\right|,$$ where $\{H_1,K_1\}$  ranges over unordered pairs of representatives of  $\mathscr{L}_{_{HK}}.$ 

\end{thm}





\begin{exa} Let $G_1$ be a finitely generated abelian group and $H\neq K$ its finite subgroups. It follows from Corollary \ref{fin ger abelian} that 
 $$|g(\hnn(G_1, H, K, f, t), G_1,  \mathcal{A})|=1.$$ Note if $G_1$ is virtually infinite cyclic with $H \neq tor(G_1) \neq K$ the result still follows by Remarks \ref{finite-by-cyclic}. 
 This result generalizes Corollary 3.3 in \cite{BZ}. 
\end{exa}

\begin{rem} \label{ki finitos}
If $k(G_1,H,K)=1$ then it follows from Theorems \ref{fixed H} and \ref{genus hnn nao fixos} that $|g(G, \F)|=|g(G, G_1,\A)|=\left|\overline{\Gamma}_{HK}  \backslash  \Omega  \right|$, where $\Omega$ was defined in (\ref{Omega como uniao}). This happens for example when  $G_1$ is finite or $H$ and $K$ are characteristic in $\widehat G_1$. 
\end{rem}

\begin{exa} \label{usar genus class A}
 Let $\Gamma$ be indecomposable into a free product finitely generated of Fuchsian groups and $H \neq \{1\}$ is one of its finite subgroup. Let $G_1=K\times \Gamma$ such that $H \cong K$. As in the example \ref{grupo fuchsian} $G_1$ is an $\Oe$-group. Consider $G=\hnn(G_1,H, K, f, t)$. It follows from the fact that $K$ is a  characteristic subgroup of $G_1$ (also of $\widehat{G}_1$) and from the calculation of the orbits made in Corollary 7.4 of \cite{BPZ2} that $k(G_1,H,K)=1$ (see Remark \ref{os ls}). It follows from the Example \ref{grupo fuchsian} and Remark \ref{ki finitos} that $|g(G, G_1,\A)|=|g(G, \F)|=1$. 
\end{exa}

\begin{exa} Let $\Gamma$ be an indecomposable into a free product finitely generated  Fuchsian group and $H \neq \{1\}$ a finite abelian group. Let $G_1=H\times \Gamma$ and $G=\hnn(G_1, H, f,t)$ a normal HNN-extension. As in the previous example  $H$ is central and characteristic in $\widehat{G}_1$ and in addition $k(G_1, H, H)=1.$ Let $\alpha \in \aut(H)$ and consider the following automorphism $\beta:G_1 \longrightarrow G_1$ that satisfies $\beta|_{_{H}}=\alpha$ and $\beta|_{_{ \Gamma}}=id,$ so we have $\alpha = \beta|_{_{H}} \in \overline{\aut}_{G_1}(H),$ from where it follows that $\widetilde{\aut}_{G_1}(H)= \out(H)=\widetilde{\aut}_{\widehat G_1}(H).$ As $H$ is central it follows that $\widetilde{N}_{\widehat{G}_1}(H)=\{1\}=\widetilde{N}_{G_1}(H).$ It follows from Remark \ref{ki finitos} and Corollary \ref{cotas n} that $$|g(G, G_1,\A)|=|g(G, \mathcal{F})| \leq \left\{\begin{array}{cc}
  1,&\mbox{ if } |\tilde f| \leq 2,\\
\displaystyle\frac{\phi(|\tilde f|)}{2}, &\mbox{ if } |\tilde f| \geq 3,
\end{array}\right.$$
where $\phi$ is Euler's function and $\tilde f$ is the image of $f$ in $\out(H)$. Furthermore, if $\aut(H)$ is abelian, the inequality becomes an equality in the formula above.
\end{exa}

\begin{exa}\label{free by Cp} Let $H$ be a group of order $3$ and $M$ a  $\Z H$-lattice of $\Z$-rank $2$ with non-trivial action. Let $G_1=M\rtimes_{\varphi} H$ be the corresponding semidirect product  and $G=\hnn(G_1,H,f,t)$ with $f\in \aut(H)$. Then $G_1$ is an $\Oe$-group (by Proposition 3.1 in \cite{BPZ2}) and has just two subgroups of order $3$ up to conjugation that can be swapped by an automorphism, so $k(G_1,H,H)=1$. 
Note that $\widehat G_1=\widehat M\rtimes_{\hat \varphi} H,$ where $\hat \varphi$ is the natural extension of $\varphi$ (see Proposition 2.6 in \cite{GZ}) and $\overline{N}_{\widehat G_1}(H) \leq \inn(H)=1$ since the action of $H$ is non-trivial, so $\widetilde{N}_{\widehat G_1}(H) = \{1\}=\widetilde{N}_{G_1}(H).$ By Lemma 2.1 \cite{Die}  $\overline{\aut}_{ G_1}(H)$ contains  $\overline N_{\aut(M)}(H)$. Besides, $\aut(M)\cong GL_2(\Z)$ contains $S_3$ and a subgroup of order $3$ is unique up to conjugation in $GL_2(\Z)$. Note that $\overline{\aut}_{G_1}(H)=\aut(H).$ Indeed, take $\alpha\in \aut(H)$ and observe that $\psi:=(id,\alpha) \in \aut(G_1)$ and $\psi|_H=\alpha \in \overline{\aut}_{G_1}(H)$. So $\widetilde{\aut}_{G_1}(H)=\widetilde{\aut}_{\widehat{G}_1}(H)=\aut(H)=\out(H)\cong C_2$. Then by Remark \ref{ki finitos} and Corollary \ref{autG classe} $|g(G,G_1,\A)|=|g(G, \mathcal{F})|=1.$ 
If the action of $H$ is trivial then $H$ is normal in $\widehat{G}_1$ and since $\aut(H) \cong C_2$, its follows that $\widetilde{N}_{G_1}(H)=\widetilde{N}_{\widehat G_1}(H)=\{1\},$  and the genus is 1. 

\end{exa}

\begin{cor} Let $G_1$ be an indecomposable into a free product finitely generated Fuchsian group and $G=\hnn(G_1, H, f, t)$ a normal HNN-extension, where $H$ is a finite subgroup of $G_1$. Then $|g(G,G_1,\A)|=|\widetilde{N}_{G_1}(H)\backslash \cR/C_2|,$ where  
$\cR=
\{\tilde f \tilde \tau_{n^{-1}}\,|\, n \in N^{+}\}$.
  In particular, if $H$ is maximal then the cardinality of the genus satisfies $$|g(G,G_1,\A)|=|\mathcal{R}/C_2| = \left\{\begin{array}{cc}
  1,&\mbox{ if } |f| \leq 2,\\
\displaystyle\frac{\phi(|f|)}{2}, &\mbox{ if } | f| \geq 3,
\end{array}\right.$$
where $\phi$ is Euler's function. 
\end{cor} 


\begin{proof} 
By Proposition 3.3 (i) in \cite{BPZ2} $G_1$ is an $\Oe$-group since indecomposable into a free product Fuchsian groups are virtually compact surface groups. By Corollary 7.4  in \cite{BPZ2} (see also the proof) we have  $l_H=1$ and so $$k(G_1, H, H)=1$$ by Remark \ref{os ls}. On the other hand, it follows from the presentation of a Fuchsian group that any automorphism of $H$ extends to an automorphism of $G_1$, i.e., $\widetilde{\aut}_{G_1}(H)=\out(H)=\widetilde{\aut}_{\widehat{G}_1}(H)$. Furthermore, $G_1$ is conjugacy separable  (see \cite{fine}), so $\overline{N}_{\widehat{G_1}}(H)=\overline{N}_{G_1}(H)$ and consequently $$\widetilde{N}_{\widehat{G_1}}(H)=\widetilde{N}_{G_1}(H).$$ Note that $\out(H)$ is abelian since $H$ is cyclic. Therefore, by Corollary \ref{nova conta conjugadoss} 
and Theorem \ref{genus hnn nao fixos} $|g(G,G_1,\A)|=|g(G,\mathcal{F})|= |\widetilde{N}_{\widehat G_1}(H)\backslash\mathcal{R}/C_2|=|\widetilde{N}_{G_1}(H)\backslash \cR/C_2|.$ If $H$ is a maximal finite subgroup then $N_{G_1}(H)=H$ and therefore $\widetilde{N}_{G_1}(H)=\{1\}=\widetilde{N}_{\widehat{G}_1}(H).$ The formula for the genus in this case follows from Corollary \ref{cotas n} and Remark \ref{H subgroup cyclic}.
\end{proof}

\begin{rem}
With hypotheses of the previous corollary, assume that $G=\hnn(G_1, H, K, f,t)$ where $H, K$ are non-conjugate finite subgroups of $G_1.$ Note that $|g(G,G_1,\A)|=|g(G,\mathcal{F})|= \left|\overline{\Gamma}_{HK}  \backslash \Omega\right|\,\,(\mbox{ see formula  } (\ref{Omega como uniao}))$ by  Theorem \ref{fixed H}. If $H$ is maximal, then $K$ is also maximal and they are self-normalized, from where it follows that $$\widetilde{N}_{G_1}(H)=\{1\}=\widetilde{N}_{G_1} (K)=\widetilde{N}_{\widehat{G}_1}(H)=\widetilde{N}_{\widehat{G}_1}(K).$$ Therefore by Proposition \ref{cp1} $\widetilde{N}_{G}(H)=\{1\}=\widetilde{N}_{\widehat{G}}(H)=\widetilde{N}^{+}.$ In this case the genus is simply given by $|g(G,G_1,\A)|=|g(G,\mathcal{F})|= \left|\overline{\Gamma}_{HK} \backslash \Omega\right|,$ where $\Omega=\overline{\Gamma}_{_{\widehat{HK}}} \cdot f.$  
\end{rem}

From Theorems \ref{genus hnn nao fixos} and \ref{1.8 da introducao} we deduce the following

\begin{prop} \label{genus classF}
Let $H$ be a finite subgroup of an $\Oe$-group $G_1$ and $$G=\hnn(G_1, H, f, t)$$ be a normal HNN-extension. Then the cardinality of genus $g(G, G_1,\A)$ satisfies $$ k(G_1, H, H) \leq |g(G, G_1,\A)| 
\leq k(G_1, H, H) \cdot |\out(H)|.$$ In particular, if $|\out(H)|=1$ then $|g(G, G_1,\A)|=k(G_1, H, H).$ 
\end{prop}




By Remarks \ref{rem dos normalizadores e densidade}, \ref{tilde implica overline}, \ref{ki finitos} and Corollary \ref{autG classe}, we have the following

\begin{cor}\label{densohnn} Let $H$ be a finite subgroup of an $\Oe$-group $G_1$ and $$G=\hnn(G_1, H, f, t)$$ be a normal HNN-extension. Suppose that  $$\widetilde{\aut}_{G_1}(H) = \widetilde{\aut}_{\widehat G_1}(H), \widetilde{N}_{\widehat G_1}(H)=\widetilde{N}_{ G_1}(H)$$ and $k(G_1, H, H)=1.$ Then $$|g(G, G_1,\A)|=|g(G, \mathcal{F})|\leq |\tilde \kappa(\widetilde{N}_{ G}(H))|.$$ In particular, if $\widetilde{N}_{\widehat G}(H)=\widetilde{N}_{G_1}(H)$ or $|\out(H)|\leq 2$ then $|g(G, G_1,\A)|=1$. 
\end{cor}

\begin{exa} Let $P=\langle x_1, x_2\rangle$ be a  free nilpotent group of rank $2$ and \break $H\cong \Z/p \Z$ be  a group of odd prime order $p$. Note that $$\aut(\Z/p \Z)=\langle c\rangle\cong C_{p-1}$$ is a  cyclic group of order  $p-1$.   We look at $c$ as the unit of $\Z/p\Z$ and define $G_1=H\rtimes_{\psi_1} P$ with $h^{x_j}= h^{c}$, for $i,j=1,2$ and by Proposition 3.1 in \cite{BPZ2} $G_1$ is an $\Oe$-group. Note that $\widehat G_1=H\rtimes_{\hat \psi_1} \widehat P,$ where $\hat \psi_1$ is the natural extension of $\psi_1$ to $\widehat P$ (see Proposition 2.6 in \cite{GZ}). Let $G=\hnn(G_1,H,f,t)$ be an HNN-extension with $f\in \aut(H)$. Since $H$ is characteristic in $\widehat G_1$, by Remark \ref{ki finitos} $k(G_1, H, H)=1$. Furthermore, $\widetilde{N}_{G_1}(H)=\out(H)=\widetilde{N}_{\widehat G_1}(H)=\widetilde{N}_{\widehat{G}}(H)$. By Example 7.11 in \cite{BPZ2} we have $\widetilde{\aut}_{ G_1}(H)=\widetilde{\aut}_{\widehat{G}_1}(H)=\aut(H)=\out(H).$ It follows from Corollary \ref{densohnn} that $|g(G,\mathcal{F})|=|g(G,G_1,\A)|=1$. 
\end{exa}

\begin{exa} Let $G_1 = GL_n(\mathbb{F}_p)$ and consider $H$ to be the Borel subgroup of the upper triangular matrices of $G_1$. Let $G=\hnn(G_1, H, f, t)$ be an HNN-extension with $f \in \aut(H).$ Note that $H$ is self-normalizing in $G_1$ so $\widetilde{N}_{G_1}(H)=\{1\},$ moreover $k(G_1, H, H)=1$ since $G_1$ is finite. By Proposition \ref{cp1} we have that $\widetilde{N}_{\widehat{G}}(H)=\left \langle \tilde{f} \right \rangle,$ where $\tilde{ f}$ is the image of $f$ in $\out(H).$ It follows from Corollaries \ref{nova conta conjugadoss}, \ref{cotas n} and  Remark \ref{ki finitos} 
that  $$|g(G,G_1,\A)|=|g(G,\mathcal{F})|= |\mathcal{R}/(\widetilde{\aut}_{ G_1}(H) \times C_2)| \leq \left| \left \langle \tilde{f} \right \rangle / (\widetilde{\aut}_{ G_1}(H) \times C_2) \right|,$$ and therefore 

$$|g(G,G_1,\A)| \leq \left\{\begin{array}{cc}
  1,&\mbox{ if } |\tilde f| \leq 2,\\
\displaystyle\frac{\phi(|\tilde f|)}{2}, &\mbox{ if } |\tilde f| \geq 3,
\end{array}\right.$$
where $\phi$ is Euler's function and $\tilde f$ is the image of $f$ in $\out(H)$.
In particular, if $f \in \inn(H),$  then $|g(G,G_1,\A)|=1.$    
\end{exa}

\begin{exa}\label{dieadra ove klein} Let $G_1= \langle r, c \ | \ r^4=c^2=1, crc^{-1}=r^{-1}\rangle \cong D_8$  and $H=\langle c, r^2\rangle \cong C_2 \times C_2$ be one of the Klein subgroups of $G_1.$ Let $f_1$ be an automorphism of $H$ that fixes $c$ and maps $r^2$ to $cr^2$ and $f_2$ an automorphism that fixes $r^2$ and maps $c$ to $cr^2$. Let $B=\hnn(G_1, H,  f_1,t_1)$ and $D=\hnn(G_1, H, f_2,t_2)$  be normal HNN-extensions. Note that $$\aut(H)=\out(H)\cong S_3, \overline{N}_{G_1}(H)=\{id, f_2=\tau_{cr}|_{_{H}}\} = \overline{\aut}_{G_1}(H) \cong C_2.$$ By Proposition \ref{nonnormalhnnabs}  $B$ and $D$ are not isomorphic 
and by Corollary \ref{corhnn}  the number of the isomorphism classes of HNN-extensions $G=\hnn(G_1, H, f,t)$ with $f\in \aut(H)$ is equal to 2 since the two distinct orbits are given by the orbits of the representatives $f_1$ and $f_2.$ By Theorem \ref{genus 1 normal}  $$|g(D, \F)|=1$$ because $f_2 \in \overline{N}_{G_1}(H)$. Consequently, $\widehat B$ and $\widehat D$
are non-isomorphic and $|g(B,\F)|=1$.
As $G_1$ is finite it follows from Remark \ref{ki finitos} that $$k(G_1, H, H)=1,$$ therefore the genus satisfies $$|g(D, \F)|=|g(B,\F)|=|g(D,G_1,\A)|=|g(B,G_1,\A)|=|g(G,G_1,\A)|=1.$$
\end{exa}

\begin{exa} In the previous example take $H=\langle r^2,  c \rangle \cong C_2 \times C_2$ and $K=\langle r^2, r c \rangle \cong C_2 \times C_2$ the distinct Klein subgroups of $G_1$ and let $$G=\hnn(G_1, H, K, f, t)$$ be a non-normal HNN-extension where $f \in \Iso(H, K)$ fixes $r^2$ and sends $c$ to $rc.$ Note that $f$ extends to an automorphism $\varphi \in \aut(G_1)$ where $\varphi(r)=r$ and $\varphi(c) = rc,$ therefore it follows from  Theorem \ref{genus 1} i) and Remark \ref{ki finitos} that $|g(G, G_1,\A)|=1.$ 

Similarly let $G_2=Q_8=\left \langle i, j\,|\, i^4=j^4=1, i^j=j^{3}=-j \right \rangle$ be the quaternion group of order 8, $H=\left \langle i \right \rangle$ and $K=\left \langle j \right \rangle $ two of its non-conjugate subgroups of order 4. Let $W=\hnn(G_2, H, K, f, t)$ such that $f(i)=\pm j.$ Note that $f$ extends to an automorphism $\psi: G_2 \longrightarrow G_2$ such that $\psi(i)=j$ and $\psi(j)=i$ or $\psi(i)=-j$ and $\psi(j)=-i.$ Then for the same reason as in the dihedral case we have  $|g(W, G_2,\A)|=1,$ and similarly for any permutation between three distinct subgroups of order 4 of the quaternion group.
\end{exa} 

\begin{exa}\label{diedral over cyclic}
In the previous example, let $H$ be any of the proper subgroups of $G_1 \cong D_8$ that are not Klein subgroups. Then $H$ is cyclic of order $\leq 4,$ from where it follows that $|\aut(H)|=|\out(H) | \leq 2,$ so $$|g(\hnn(D_8, H, H, f,t), D_8,\A)|=1$$ by Corollary \ref{autG classe} and Remark \ref{ki finitos}.  Therefore, for any HNN-extension $$G=\hnn(D_8, H, H, f,t)$$ the cardinality of  genus $g(G, G_1,\A)$ is always equal to 1 regardless of the chosen associated subgroup. If $G_2=Q_8$ is the Quaternion group of order $8,$ the same follows because all its proper subgroups are cyclic of order $\leq 4,$ i.e., $$|g(\hnn(Q_8, H, H, f,t), Q_8,\A)|=1$$ independently of the chosen associated subgroups.     
\end{exa}

\begin{rem} Note that $$G=\hnn(D_8, H, K, f,t),\  {\rm and}\  G=\hnn(Q_8, H, K, f,t)$$ are  profinitely rigid if the Remeslennikov's question has a positive answer. Indeed,  a finitely generated residually finite group $B$  having the same profinite completion as $G$ would be virtually free and by \cite[Theorem 1.1]{BPZ} $B$ will be an HNN-extension with finite base group.
\end{rem}

\subsection{Genus in a more general class} \label{general section}

 In this subsection we give formulas  for the genus of HNN-extensions in class $\mathcal{A}$ with $G_1$ not fixed. When $G$ is a group in $\mathcal{A}$,  we will denote the genus of $G$ in class $\mathcal{A}$, as $g(G, \mathcal{A}).$

Let $H$ and $K$ be finite subgroups of an $\Oe$-group $G_1$ and $$G=\hnn(G_1, H, K, f, t)$$ an HNN-extension. The objective  is to find a formula to calculate the cardinality of the set $g(G, \mathcal{A}).$
To do this, consider $B\in \mathcal{A}$ such that $\widehat B\cong \widehat G$. Then by \cite[Theorem 1.1]{BPZ} $B$ is an HNN-extension of the form $B=\hnn(B_1, H_1, K_1, f_1, t_1)$ with $\widehat B_1\cong \widehat G_1, H \cong H_1\cong K \cong K_1,$ and furthermore $B_1 \in \mathcal{A}$.


 \begin{deff} Let $G_1$ be a group and $\widehat G_1$ be its profinite completion. Let $X_1, X_2, Y_1, Y_2$ be finite subgroups of $\widehat G_1$. We say that the pairs $(X_1,Y_1)$ and $(X_2,Y_2)$ are \textbf{permutable}, if there exists $\psi\in\aut(\widehat G_1)$ that  permutes the conjugacy classes of $X_i$ and $Y_i,$ for each  $i=1,2$.\end{deff} 
 
Let $B_1 \in g(G_1,\A)$. We will identify $\widehat G_1$ with $\widehat B_1$ and consider $G_1$ and $B_1$ as dense subgroups of $\widehat G_1=\widehat B_1.$ 
  Let $B=\hnn(B_1,H_1,K_1,f_1,t_1)$ be an HNN-extension with finite associated subgroups. We say that $B$ is \textbf{ compatible }  with $G$ if the pairs $(H,H_1)$ and $(K,K_1)$ are permutable in $\widehat G_1$ (i.e. exists $\psi\in\aut(\widehat G_1)$ such that $\psi(H)=H_1$ and $\psi(K)=K_1^{g_1}$ for some $g_1\in \widehat G_1$) and $f_1 = \tau_{g_1} f^{\psi^{-1}}$.
By Proposition \ref{2a}, $B$ is compatible with $G$ if and only if $\widehat B\cong \widehat G.$ 
With these considerations, we state the following result

\begin{thm}\label{fixed H2} The cardinality of the genus $g(G, \mathcal{A})$ of an HNN-extension $G=\hnn(G_1, H, K, f, t)$ is equal to $$|g(G,\A)| = \sum_{B_1\in g(G_1,\A)} |g\left(\hnn(B_1, H_1, K_1, f_1, t_1  ), B_1,\A\right)|,$$ where $B_1$ ranges over the representatives of isomorphism classes in $g(G_1,\A)$ such that $G$ is compatible with $\hnn(B_1, H_1, K_1, f_1, t_1 ).$   
\end{thm}


\begin{cor} \label{genus 1 gerall}
If $| g(G_1,\A)|=1$ then $|g(G, \mathcal{A})|=|g(G, G_1,\A)|.$    
\end{cor}

\begin{exa} 
 Let $\Gamma$ be an indecomposable into a free product finitely generated  Fuchsian group and $H \neq \{1\}$ is one of its finite subgroup. Let $G_1=K\times \Gamma$ such that $H \cong K$. Consider $G=\hnn(G_1,H, K, f, t)$. If $\Gamma$ is profinitely rigid (see \cite{BMRS20} for examples of such) then so is $G_1$ and by Example \ref{usar genus class A} and Corollary \ref{genus 1 gerall}  $|g(G,\mathcal{A})|=1.$  
\end{exa}


\end{document}